\documentclass[a4paper,11pt]{article}
\usepackage{amsmath,amsthm,amsfonts,amssymb,bbm}
\usepackage{graphicx,psfrag,subfigure,cite}
\usepackage[usenames,dvipsnames]{xcolor}
\usepackage{tikz-cd} 
\usepackage{tikz}
\usepackage[normalem]{ulem}
\usepackage[colorlinks]{hyperref} 
\DeclareFontFamily{U}{BOONDOX-calo}{\skewchar\font=45 }
\DeclareFontShape{U}{BOONDOX-calo}{m}{n}{
  <-> s*[1.05] BOONDOX-r-calo}{}
\DeclareFontShape{U}{BOONDOX-calo}{b}{n}{
  <-> s*[1.05] BOONDOX-b-calo}{}
\DeclareMathAlphabet{\mathcalboondox}{U}{BOONDOX-calo}{m}{n}
\SetMathAlphabet{\mathcalboondox}{bold}{U}{BOONDOX-calo}{b}{n}
\DeclareMathAlphabet{\mathbcalboondox}{U}{BOONDOX-calo}{b}{n}

\usepackage[charter]{mathdesign}

\usepackage{xcolor}
\definecolor{brightcerulean}{rgb}{0.11, 0.67, 0.84}
\definecolor{cerulean}{rgb}{0.0, 0.48, 0.65}
\definecolor{Gray}{rgb}{0.5, 0.5, 0.5}
\definecolor{columbiablue}{rgb}{0.61, 0.87, 1.0}
\definecolor{OliveGreen}{rgb}{0, 0.26, 0.15}
\definecolor{dark red}{rgb}{0.7, 0., 0.2}
\definecolor{dark red 2}{rgb}{0.81, 0.09, 0.13}
\definecolor{dark blue}{rgb}{0, 0.18, 0.39}
\definecolor{dark blue 2}{rgb}{0.03, 0.27, 0.49}
\definecolor{dark green2}{rgb}{0.07, 0.53, 0.03}
\definecolor{dark green}{rgb}{0, 0.44, 0}
\definecolor{lightcarminepink}{rgb}{0.9, 0.4, 0.38}
\definecolor{lightcoral}{rgb}{0.94, 0.5, 0.5}
\definecolor{lightcornflowerblue}{rgb}{0.6, 0.81, 0.93}
\definecolor{lightcyan}{rgb}{0.88, 1.0, 1.0}
\definecolor{lavenderblush}{rgb}{1.0, 0.94, 0.96}
\definecolor{deeppeach}{rgb}{1.0, 0.8, 0.64}
\definecolor{darkchampagne}{rgb}{0.76, 0.7, 0.5}
\definecolor{desertsand}{rgb}{0.93, 0.79, 0.69}
\definecolor{classicrose}{rgb}{0.98, 0.8, 0.91}
\definecolor{myyeallow}{rgb}{0.98, 0.91, 0.71}
\definecolor{carolinablue}{rgb}{.1, .7, 1}
\definecolor{antiquewhite}{rgb}{0.98, 0.92, 0.84}
\definecolor{mycolor}{rgb}{0.98, 0.91, 0.71}
\definecolor{darkerblue}{rgb}{0.45, 0.7, 1.0}
\definecolor{limes}{rgb}{0.2, .9, 0.2}
\definecolor{softgreen}{rgb}{0.8, 1, 0.8}
\definecolor{softblue}{rgb}{.9, .9, 1}
\definecolor{camel}{rgb}{0.76, 0.6, 0.42}
\definecolor{brightcerulean}{rgb}{0.11, 0.67, 0.84}
\definecolor{cerulean}{rgb}{0.0, 0.48, 0.65}
\definecolor{Gray}{rgb}{0.5, 0.5, 0.5}
\definecolor{blizzardblue}{rgb}{0.67, 0.9, 0.93}
\definecolor{cyan}{rgb}{0.0, 1.0, 1.0}
\definecolor{mauve}{rgb}{0.86, 0.82, 1.0}
\definecolor{paleblue}{rgb}{0.2, 0.85, 0.97}
\definecolor{bronze}{rgb}{0.8, 0.5, 0.2}
\definecolor{nogreen}{rgb}{.9, 0.1, 0.3}
\definecolor{newblue}{rgb}{.1,.4,1}
\numberwithin{equation}{section}

\newcommand{\nur}{\nu^N_{\rho(\cdot)}}

\newcommand{\ola}[1]{\overleftarrow{#1}}
\newcommand{\ora}[1]{\overrightarrow{#1}}

\newcommand{\Pb}{\mathbb{P}}
\newcommand{\E}{\mathbb{E}}

\newcommand{\R}{\mathbb{R}}
\newcommand{\N}{\mathbb{N}}
\newcommand{\Z}{\mathbb{Z}}

\newcommand{\Or}{{\cal O}}

\let\emptyset\varnothing

\newlength\squareheight
\newtheorem{prop}{Proposition}[section]
\newtheorem{thm}[prop]{Theorem}
\newtheorem{lem}[prop]{Lemma}
\newtheorem{defin}[prop]{Definition}
\newtheorem{cor}[prop]{Corollary}

\newtheorem{cla}[prop]{Claim}

\newtheorem{rem}[prop]{Remark}
\newenvironment{remark}{\begin{rem}\normalfont}{\end{rem}}

\newcommand{\mcb}[1]{{\mathcalboondox #1}}
\definecolor{amethyst}{rgb}{0.6, 0.4, 0.8}

\title{Hydrodynamics for the ABC model with slow/fast boundary}
\author{
 Patricia Gon\c{c}alves \\
 \texttt{pgoncalves@tecnico.ulisboa.pt}
 \and
 Ricardo Misturini \\
 \texttt{ricardo.misturini@ufrgs.br}
 \and
  Alessandra Occelli \\
  \texttt{alessandra.occelli@ens-lyon.fr}
}
\date{\today}

\allowdisplaybreaks

\begin{document}
\maketitle

\begin{abstract}
In this article, we consider the ABC model in contact with slow/fast reservoirs. In this model, there is at most one particle per site, which can be of type  $\alpha\in\{A,B,C\}$ and particles exchange positions in the discrete set of points $\{1,\cdots, N-1\}$ with a weakly asymmetric rate that depends on the type of particles involved in the exchange mechanism. At the boundary points $x=1, N-1$ particles can be injected or removed with a rate that depends on the type of particles involved. We prove that the hydrodynamic limit, in the diffusive time scale, is given by a system of non-linear coupled equation with several boundary conditions, that depend on the strength of the reservoir's action. 

\end{abstract}

\section{Introduction}\label{sec:intro}

The ABC model, introduced by Evans el al. \cite{ekkm1-1998,ekkm2-1998}, is a system consisting of three species of particles, labeled  $A$, $B$, and $C$. In its original and most studied version, the particles are located on a one-dimensional discrete torus with $N$ points (one particle per site). The system evolves under nearest neighbor transpositions: $AB\to BA$, $BC\to CB$, $CA\to AC$ with rate $q\leq1$ and $BA\to AB$, $CB\to BC$, $AC\to CA$ with rate $1/q$. In particular, for this dynamics defined on the torus, the total number of particles of each species, $N_A$, $N_B$ and $N_C$, is conserved and $N_A+N_B+N_C=N$. The invariant measure is explicitly computed only when $N_A=N_B=N_C$, in which case the dynamics is reversible with respect to the Gibbs measure of a certain Hamiltonian having \emph{long range} pair interactions. When $q<1$, the bias in the dynamics drives, say, a $B$ particle to move to the right when it is inside a region of $A$ particles and to the left when it is inside a region of $C$ particles. Therefore starting from an arbitrary configuration the system will reach, after a relatively short time, a metastable state of the type $\ldots AABBBBCCCAAABB\ldots$ with blocks of particles (pure regions) located in the cyclic order $ABC$. As discussed in \cite{ekkm1-1998,ekkm2-1998}, in the thermodynamic limit $N\to\infty$, with $N_\alpha/N \to r_\alpha$ for $\alpha\in\{A,B,C\}$, and $q<1$ constant, the system segregates into three pure $A$, $B$ and $C$ clusters. In \cite{m2016}, in a proper time scale, the asymptotic dynamics among these stable configurations, which differ only by rotation, was studied in the strongly asymmetric regime, where $q\to0$ when $N\to\infty$.

While there is always strong separation when $q<1$ is held fixed, particles exchange symmetrically when $q=1$, in which case all configurations are equally probable. So, a natural scaling to investigate, introduced by Clincy et al. \cite{cde2003}, is the \textit{weakly asymmetric} regime where $q=e^{-\beta/(2N)}$, $N\to\infty$ and $\beta>0$ is a fixed control parameter which plays the role of the inverse temperature. In this regime, they observed phase transitions in the invariant measure $\tilde \mu^{\beta}_N$, at a critical value $\beta_c$. For equal densities $(r_A=r_B=r_C=1/3)$ it is known that $\beta_c=2\pi\sqrt3$. This phase transition, also analysed in \cite{bdlw2008} and \cite{aclmms2009}, can be stated in terms of the  free energy functional $\mcb{F}_\beta$ associated to $\tilde\mu^\beta_N$. The typical configurations of the system as $N\to\infty$ are described by the density profiles that minimize $\mcb{F}_\beta(\rho^A, \rho^B, \rho^C)$, where for $\alpha\in\{A,B,C\}$, $\rho^\alpha$ denotes the density of particles of type $\alpha$. For $\beta<\beta_c$ (disordered high temperature phase) $\mcb{F}_\beta$ has an unique minimizer; for $\beta>\beta_c$ (segregated low temperature phase) $\mcb{F}_\beta$ has a continuum of minimizers, parameterized by translation. As discussed in \cite{aclmms2009} this means that in the limit $N\to\infty$,  $N_\alpha/N \to r_\alpha$, $\alpha\in\{A,B,C\}$, for $\beta<\beta_c$ the invariant measure $\tilde \mu^{\beta}_N$ converges in the sense of Definition \ref{def:measure:associated} to the deterministic flat profile $(\rho^A,\rho^B, \rho^C)=(r_A,r_B,r_C)$, while for $\beta>\beta_c$ the limiting profile in the sense of convergence in distribution, for $\tilde \mu^\beta_N$,  is random and space-dependent. 

In \cite{aclmms2009} the model was study on a one-dimensional lattice with $N$ sites with closed (zero flux) boundary. The dynamics is the same as above, except that a particle at $x=1$ (respectively $x=N$) can only jump to the right (respectively left).
In contrast with the dynamics defined on the torus, for the system with this boundary mechanism, independently from the values of $N_A$, $N_B$ and $N_C$, the dynamics always satisfies detailed balance with respect to a Gibbs measure $\tilde \mu^\beta_N(\cdot)=Z_{N,\beta}^{-1}e^{-\beta H_N(\cdot)}$, where  $H_N(\cdot)$ is an explicit Hamiltonian function of the configuration. It coincides with the invariant measure on the torus when $N_A=N_B=N_C$. 

As discussed in \cite{cde2003}, \cite{bdlw2008} and \cite{bcp2011}, for the dynamics on the torus, in the weakly asymmetric regime $q=e^{-\beta/2N}$, under the diffusive time scale $tN^2$ the density profiles of the three species $(\rho^A_t,\rho^B_t,\rho^C_t)$ evolve according to the system of hydrodynamic equations
\begin{equation}\label{hidro_toro}
\begin{cases}
\partial_t \rho^A=\Delta \rho^A+\beta\nabla[\rho^A(\rho^B-\rho^C)],\\
\partial_t \rho^B=\Delta \rho^B+\beta\nabla[\rho^B(\rho^C-\rho^A)],\\
\partial_t \rho^C=\Delta \rho^C+\beta\nabla[\rho^C(\rho^A-\rho^B)],\\
\end{cases}
\end{equation}
where $\Delta$ and $\nabla$ denote, respectively, the Laplacian and gradient operators on the continuous torus.
One can make the equivalent choice of interpreting particles of species $C$ as holes, denoted by $\emptyset$ and looking only at the evolution of the density profiles for $A$ and $B$. As a consequence of the exclusion rule, $\rho^A(x)+\rho^B(x)+\rho^\emptyset(x)=1$ for any $x\in[0,1]$, the hydrodynamic equations read as
\begin{equation*}
\begin{cases}
\partial_t \rho^A=\Delta \rho^A+\beta\nabla[\rho^A(\rho^A+2\rho^B-1)]\\
\partial_t \rho^B=\Delta \rho^B-\beta\nabla[\rho^B(2\rho^A+\rho_B-1)].
\end{cases}
\end{equation*}
The hydrodynamic equations for the weakly asymmetric model on the torus have been formally derived in~\cite{master}.

In the present work, we study a version of the weakly asymmetric ABC model on the interval $\{1,\ldots, N-1\}$ with a boundary dynamics  that allow the exchange of particles with reservoirs on both sides of the bulk, with a strength that can be regulated by some parameters that increase or decrease the reservoirs' action. The precise definition of the model is given in the next section. For this model we prove the hydrodynamic limit for the density of each type of particle, whose hydrodynamic equation is given by \eqref{hidro_toro} and with Dirichlet or Robin boundary conditions,  that depend on the rates of exchange of particles with the reservoirs. In this  non-conservative model, unlike the case considered in \cite{aclmms2009}, the invariant measure is not reversible and it is not explicitly known.

Our result is an extension of what was considered for the case of one-species in \cite{BMNS17} and \cite{capitao}.
The strategy of proof we employ is based on the entropy method introduced in \cite{GPV}. The idea of the method  results of combining the tightness of the sequence of empirical measures with the unique characterization of the limit points, which gives convergence of the whole sequence. This limit point is supported on a trajectory of measures,  absolutely continuous with respect to the  Lebesgue measure and whose density is the unique weak solution of the corresponding hydrodynamic equation. Our model has  weakly asymmetric  rates and a boundary dynamics. These two features bring at the macroscopic level a system of non-linear equations with several boundary conditions. To prove the hydrodynamic limit we need to obtain several replacement lemmas. One is necessary at the bulk in order to overcome the fact that the dynamic is weakly asymmetric and one needs to make a proper linearisation at the microscopic level. Other replacements are needed at the boundary so that one recognizes the boundary conditions of Dirichlet and Robin type. This last replacement, which consists of replacing occupations variables close to the boundary by the respective density reservoirs,  is crucial, and we obtain it by making a comparison with the action of the adjoint operator at those occupation variables.

There are many intriguing questions that we leave for future work. The first one has to do with the derivation of the hydrostatic limit, i.e. the hydrodynamic limit starting from the stationary measure of the system. The second one is related to  the derivation of the fluctuations around the hydrodynamical profile. Since the equations are coupled, determining the correct fluctuation fields and the limiting equations is  a very interesting problem. Our method could also be employed to other variations of the dynamics, as e.g. considering $M$-types of particles and allowing a finite number of particles per site.

Here follows an outline of this work. In Section \ref{sec:model} we introduce the model and we state our main result. In Section \ref{sec:heuristics} we give an heuristic argument which allows recognizing the limit equation and the respective boundary conditions, depending on the boundary strength. Section \ref{sec:tightness} is devoted to the proof of tightness of the sequence of  empirical measures associated to each type of particle. In Section \ref{sec:charac} we rigorously characterize the limit point and Section \ref{sec:RL} is devoted to the proof of the replacement lemmas that are needed along the proofs. In Appendix \ref{sec:uniqueness} we prove the uniqueness of weak solutions of the equations we derive.

\section{The ABC model with slow/fast boundary}\label{sec:model}

Let $\theta\geq \delta$ and $\theta\geq 1$, $\beta,\tilde\beta>0$ 
and $r_{AB},r_{B\emptyset},r_{\emptyset A},r_{BA},r_{\emptyset B},r_{A\emptyset}>0$. Let $N\in\N$. We consider the one-dimensional two-species weakly asymmetric simple exclusion process with state space $\Omega_N=\{A,B,\emptyset\}^{\Lambda_N}$, where $\Lambda_N = \{1, \dots, N-1\}$. A configuration in $\Omega_N$ is denoted by $\eta$, where $\eta(x)=\alpha$ if site $x$ is occupied by a particle of type $\alpha$. We make the convention that $\alpha+1,\,\alpha+2,\ldots$ denote the particle types that are successors to $\alpha$ in the cyclic order $AB\emptyset$. For $\alpha\in\{A,B,\emptyset\}$, and $x\in\Lambda_N$ we denote by $\xi^\alpha_x:\Omega_N\to\{0,1\}$ the function that indicates that there is a particle of type $\alpha$ at site $x$, i.e, $\xi^\alpha_x(\eta)=1$ if $\eta(x)=\alpha$ and $\xi^\alpha_x(\eta)=0$ otherwise. As discussed in the previous section, this model is known in literature as the ABC model, in which case, holes $\emptyset$ are described as particles of type $C$. The dynamics is the following: a couple of particles $(\eta(x),\eta(x+1))=(\alpha,\alpha+1)$ swaps positions with rate $1-{\frac{\beta}{2N}}$, while the configuration $(\eta(x),\eta(x+1))=(\alpha+1,\alpha)$ is inverted with rate $1+{\frac{\beta}{2N}}$. This implies that particles prefer to arrange themselves in cyclic alphabetical order. At the boundary, particles are injected or removed with the same behaviour, but the swapping depends on different rates $r_{\alpha \beta}$ according to the relative positions of the particle at $x=1,N-1$ and at the corresponding reservoirs. 

The infinitesimal generator $\mcb L_N$ of the Markov process acts on functions $f : \Omega_N \to \R$ as $\mcb L_N =\mcb L_N^L + \mcb L_N^B +\mcb L_N^R$, where the central term corresponds to the dynamics of the bulk, while the terms on the left and on the right are the generators corresponding to the left and the right boundary, respectively. The bulk generator is
\begin{equation*}
 \mcb L_N^B f(\eta)=\sum_{x=1}^{N-2}c_{x,x+1}(\eta)[f(\eta^{x,x+1})-f(\eta)]
\end{equation*}
where $\eta^{x,x+1}$ is the configuration obtained from $\eta$ by exchanging particles at sites $x$ and $x+1$, i.e,
\begin{equation*}
\eta^{x,x+1}(y)=\begin{cases}
\eta(y),& \text{ if } y\notin \{x,x+1\},\\
\eta(x+1) & \text{ if } y=x,\\ 
\eta(x) & \text{ if } y=x+1,
\end{cases}
\end{equation*}
and
\begin{equation*}
 c_{x,x+1}(\eta)=\begin{cases}
                  1+{\frac{\beta}{2N}} & \text{if } (\eta(x),\eta(x+1))\in\{(B,A),(\emptyset,B),(A,\emptyset)\},\\
                  1-{\frac{\beta}{2N}} & \text{if } (\eta(x),\eta(x+1))\in\{(A,B),(B,\emptyset),(\emptyset,A)\}.
                 \end{cases}
\end{equation*}
The boundary generators are given by
\begin{equation}\label{eq:boundary_generator}
\begin{aligned}
 \mcb L_N^L f(\eta)= & c^+_{1}(\eta)[f(\eta^{1,+})-f(\eta)]+c^-_{1}(\eta)[f(\eta^{1,-})-f(\eta)]\\
 \mcb L_N^R f(\eta)= & c^+_{N-1}(\eta)[f(\eta^{N-1,+})-f(\eta)]+ c^-_{N-1}(\eta)[f(\eta^{N-1,-})-f(\eta)],
 \end{aligned} 
\end{equation}
where 
\begin{equation*}
 \begin{aligned}
   & c^+_{1}(\eta)=\big(\frac{1}{N^\delta}+\frac{\tilde\beta}{2N^\theta}\big)\big[\xi^A_1(\eta)r_{AB}+\xi^B_1(\eta)r_{B\emptyset}+\xi^\emptyset_1(\eta)r_{\emptyset A}\big],\\
  &  c^-_{1}(\eta)=\big(\frac{1}{N^\delta}-\frac{\tilde\beta}{2N^\theta}\big)\big[\xi^A_1(\eta)r_{A\emptyset}+\xi^B_1(\eta)r_{BA}+\xi^\emptyset_1(\eta)r_{\emptyset B}\big],\\
 &  c^+_{N-1}(\eta)=\big(\frac{1}{N^\delta}-\frac{\tilde\beta}{2N^\theta}\big)\big[\xi^A_{N-1}(\eta)\tilde{r}_{AB}+\xi^B_{N-1}(\eta)\tilde{r}_{B\emptyset}+\xi^{\emptyset}_{N-1}(\eta)\tilde{r}_{\emptyset A}\big],\\
  & c^-_{N-1}(\eta)=\big(\frac{1}{N^\delta}+\frac{\tilde\beta}{2N^\theta}\big)\big[\xi^A_{N-1}(\eta)\tilde{r}_{A\emptyset }+\xi^B_{N-1}(\eta)\tilde{r}_{BA}+\xi^\emptyset_{N-1}(\eta)\tilde{r}_{\emptyset B}\big].
 \end{aligned}
\end{equation*}
and $\eta^{x,\pm}$ is obtained from $\eta$ when at site $x$ a particle $\alpha$ is replaced by $\alpha\pm1$, i.e,
\begin{equation*}
\eta^{x,\pm}(y)=\begin{cases}\eta(y) & \text{ if }y\neq x, \\ \eta(x)\pm1 & \text{ if } y=x. \end{cases}
\end{equation*}
Observe that when $\theta > \delta$ the rates given above are well defined for $N$ sufficiently large, but in the case $\theta=\delta$ we need to impose the restriction \begin{equation}\label{eq:rest_beta_tilde}
\tilde\beta< 2.\end{equation} 
A simple computation shows that 
\begin{equation}\label{confusion}
\begin{aligned}
\xi^A_x(\eta^{x,+})=\xi^\emptyset_x(\eta),\quad \xi^B_x(\eta^{x,+})=\xi^A_x(\eta),\quad \xi^\emptyset_x(\eta^{x,+})=\xi^B_x(\eta),\\
\xi^A_x(\eta^{x,-})=\xi^B_x(\eta),\quad \xi^B_x(\eta^{x,-})=\xi^\emptyset_x(\eta),\quad \xi^\emptyset_x(\eta^{x,-})=\xi^A_x(\eta). 
\end{aligned}
\end{equation}
To understand these relations, let us look at the top and leftmost one in \eqref{confusion}. The following chain of equivalences holds:
\begin{equation*}
\xi^A_x(\eta^{x,+})=1\iff \eta^{x,+}(x)=A\iff \eta(x)=\emptyset\iff \xi^\emptyset_x(\eta)=1.
\end{equation*}

Let us now describe the dynamics at the boundary. For example, let us look at the leftmost term on the  first line of \eqref{eq:boundary_generator}. If there is a particle of species $A$ at $x=1$, then this particle will be replaced by a particle of species $B$ with rate $r_{AB}\big(\frac{1}{N^\delta}+\frac{\tilde{\beta}}{2N^\theta}\big)$ or by a hole/it will disappear with rate $r_{A\emptyset}\big(\frac{1}{N^\delta}-\frac{\tilde{\beta}}{2N^\theta}\big)$.
For $x=1$ or $x=N-1$, we can interpret $c_x^+(\eta)$ as the rate for upgrading a particle at site $x$ and $c_x^-(\eta)$ as the rate for downgrading.

If we impose the following relationships on the rates
\begin{equation}\label{r_AB=r_B}
\begin{aligned}
	r_{AB}=r_{\varnothing B} &= r_B, & \qquad\tilde r_{AB}=\tilde r_{\varnothing B} &= \tilde r_B\\
	r_{B\varnothing}=r_{A\varnothing} &= r_\varnothing, & \qquad  \tilde r_{B\varnothing}=\tilde r_{A\varnothing} &= \tilde r_\varnothing\\
	r_{\varnothing A}=r_{BA} &= r_A,  &\qquad \tilde r_{\varnothing A}=\tilde r_{BA} &= \tilde r_A\\
	\end{aligned}
\end{equation}
with
\begin{equation}\label{relations_of_r_AB...}
r_A+r_B+r_\emptyset=1  \quad\qquad \tilde r_A+\tilde r_B+\tilde r_\emptyset=1,
\end{equation}
we can interpret $r_A$, $r_B$, $r_\emptyset$ as the concentration of particles of types $A$, $B$ and $\emptyset$ at the left reservoir, and $\tilde r_A$, $\tilde r_B$, $\tilde r_\emptyset$ as the concentrations at the right  reservoir. With this choice of rates the boundary generators become
\begin{equation}\label{eq:Lgen}
\begin{aligned}
\mcb L_N^L f(\eta)& =\big(\frac{1}{N^\delta}+\frac{\tilde\beta}{2N^\theta}\big)(\xi^A_1(\eta)r_{B}+\xi^B_1(\eta)r_{\emptyset}+\xi^\emptyset_1(\eta)r_{A})[f(\eta^{1,+})-f(\eta)]\\ 
&+\big(\frac{1}{N^\delta}-\frac{\tilde\beta}{2N^\theta}\big)(\xi^A_1(\eta)r_{\emptyset}+\xi^B_1(\eta)r_{A}+\xi^\emptyset_1(\eta)r_{B})[f(\eta^{1,-})-f(\eta)],
\end{aligned}
\end{equation}
\begin{equation*}
\begin{aligned}
\mcb L_N^R f(\eta)&=\big(\frac{1}{N^\delta}-\frac{\tilde\beta}{2N^\theta}\big) (\xi^A_{N-1}(\eta)\tilde{r}_{B}+\xi^B_{N-1}(\eta)\tilde{r}_{\emptyset}+\xi^{\emptyset}_{N-1}(\eta)\tilde{r}_{A}) [f(\eta^{N-1,+})-f(\eta)]\\
&+ \big(\frac{1}{N^\delta}+\frac{\tilde\beta}{2N^\theta}\big) (\xi^A_{N-1}(\eta)\tilde{r}_{\emptyset}+\xi^B_{N-1}(\eta)\tilde{r}_{A}+\xi^\emptyset_{N-1}(\eta)\tilde{r}_{B}) [f(\eta^{N-1,-})-f(\eta)].
\end{aligned} 
\end{equation*}

The dynamics can be summarized in the figure below.

\begin{figure}[h]
\centering
 \begin{tikzpicture}
  \draw[line width=.35mm] (0,0) -- (11,0);
  \draw[line width=.35mm] (0,-.1) -- (0,0.1);
  \draw[line width=.35mm] (1,-.1) -- (1,0.1);
  \draw[line width=.35mm] (2,-.1) -- (2,0.1);
  \draw[line width=.35mm] (3,-.1) -- (3,0.1);
  \draw[line width=.35mm] (4,-.1) -- (4,0.1);
  \draw[line width=.35mm] (5,-.1) -- (5,0.1);
  \draw[line width=.35mm] (6,-.1) -- (6,0.1);
  \draw[line width=.35mm] (7,-.1) -- (7,0.1);
  \draw[line width=.35mm] (8,-.1) -- (8,0.1);
  \draw[line width=.35mm] (9,-.1) -- (9,0.1);
  \draw[line width=.35mm] (10,-.1) -- (10,0.1);
  \draw[line width=.35mm] (11,-.1) -- (11,0.1);
  \shade [ball color=yellow] (0,0.3) circle [radius=.25cm];
  \shade [ball color=yellow] (0,0.8) circle [radius=.25cm];
  \shade [ball color=lightcyan] (0,1.3) circle [radius=.25cm];
  \shade [ball color=orange] (0,1.8) circle [radius=.25cm];	
  \shade [ball color=lightcyan] (11,0.3) circle [radius=.25cm];
  \shade [ball color=orange] (11,0.8) circle [radius=.25cm];
  \shade [ball color=orange] (11,1.3) circle [radius=.25cm];
  \shade [ball color=yellow] (11,1.8) circle [radius=.25cm];		
  \shade [ball color=lightcyan] (11,2.3) circle [radius=.25cm];
  \shade [ball color=lightcyan] (1,0.3) circle [radius=.25cm];\shade [ball color=yellow] (2,0.3) circle [radius=.25cm];
  \shade [ball color=yellow] (3,0.3) circle [radius=.25cm];
  \shade [ball color=orange] (4,0.3) circle [radius=.25cm];
  \shade [ball color=lightcyan] (5,0.3) circle [radius=.25cm];
  \shade [ball color=lightcyan] (6,0.3) circle [radius=.25cm];
  \shade [ball color=orange] (7,0.3) circle [radius=.25cm];
  \shade [ball color=yellow] (8,0.3) circle [radius=.25cm];
  \shade [ball color=lightcyan] (9,0.3) circle [radius=.25cm]; 
  \shade [ball color=yellow] (10,0.3) circle [radius=.25cm];
  \draw[->,line width=.35mm] (0,2.2) to [out=20,in=80] (1,1.6);
  \node[font=\large] at (0.5,2.7){ $r_B(\tfrac{1}{N^\delta}-\tfrac{\tilde\beta}{2N^\theta})$};
  \draw[->,line width=.35mm] (1,-.2) to [out=-120,in=-80] (-.7,-.2);
  \draw[<-,line width=.35mm] (10,2) to [out=110,in=140] (11,2.6);
  \node[font=\large] at (10.5,3.1){ $\tilde r_\emptyset(\tfrac{1}{N^\delta}+\tfrac{\tilde\beta}{2N^\theta})$};
  \draw[->,line width=.35mm] (10,-.2) to [out=-80,in=-120] (11.7,-.2);
  \draw[<->,line width=.35mm] (3,.6) to [out=60,in=110] (4,.6);
  \node[font=\large] at (3.5,1.2){ $1-\tfrac{\beta}{2N}$};
  \draw[<->,line width=.35mm] (6,.6) to [out=60,in=110] (7,.6);
  \node[font=\large] at (6.5,1.2){ $1+\tfrac{\beta}{2N}$};
  \draw[<->,line width=.35mm] (8,.6) to [out=60,in=110] (9,.6);
  \node[font=\large] at (8.5,1.2){ $1+\tfrac{\beta}{2N}$};

 \end{tikzpicture}
 \caption{Dynamics of the ABC model with reservoirs at $x=0,N$. Particles of species $A$ in yellow, $B$ in orange and $C$/holes in light cyan.}
\end{figure}
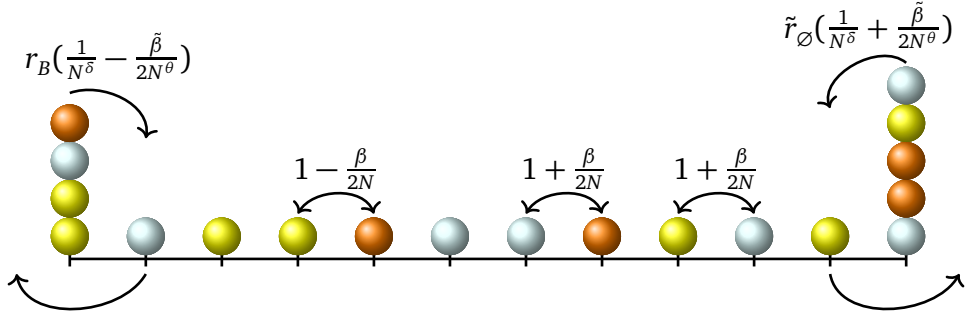

\subsection{Hydrodynamic equations }\label{sec:hydroeq}

{Our goal in this article is to describe the space-time evolution of the density of each species of particles as solutions of a system of partial differential equations (PDEs).}

The solutions that we obtain are in the  weak sense and to define them properly, we first need to define the set of functions that we test the solutions with. To that end,  for $m,n\in\mathbb N_0$,  let $C^{m,n}([0, T] \times [0,1])$ be the set of continuous functions defined on $[0, T] \times [0,1] $ that are $m$ times differentiable on the first variable and $n$ times differentiable on the second variable, and with continuous derivatives. For a function $\phi\in C^{m,n}([0, T] \times [0,1])$, we denote by $\partial_t\phi$ its derivative with respect to the time variable $t$ and by $\nabla \phi$ and $\Delta \phi$ its first and second derivatives with respect the the space variable. We use $\phi_t(u)$ as a notation to $\phi(t,u)$. We also denote  $C^{m,n}_c ([0,T] \times [0,1])$ as the set of  functions $G \in C^{m,n}([0, T] \times[0, 1])$ such that for each time $s$,  $G_s$ has a compact support included in $(0,1)$ and we denote by $C_{c}^{m}(0,1)$ (resp. $C_c^\infty (0,1)$) the set of all $m$ continuously differentiable (resp. smooth) real-valued  functions defined on $(0,1)$ with compact support. 
We denote by $\langle\cdot,\cdot\rangle$ the inner product on $L^2([0,1])$ and by $\langle\cdot,\cdot\rangle_\mu$ the same inner product with respect to a measure $\mu$. 
\begin{defin}
Let $\mcb H^1(0, 1)$ be the set of all locally integrable functions $g : (0, 1) \to\R$ such that there exists a function $\partial_u g\in L^2(0,1)$ satisfying
$
 \langle \partial_u \varphi,g\rangle=-\langle\partial_u g,\varphi\rangle,
$
for all $\varphi\in C^\infty_c((0,1))$. For $g \in \mcb H^1(0,1)$ we define the norm by
\begin{equation*}
 \|g\|^2_{\mcb H^1(0,1)}=\|g\|^2_{L^2(0,1)}+\|\partial_u g\|^2_{L^2(0,1)}.
\end{equation*}
\end{defin}
The elements in $\mcb H^1(0,1)$ coincide a.e. with absolutely continuous functions whose derivative (which exist a.e.) belong to $L^2(0,1)$.
\begin{defin}
Let $L^2([0,T]\times \mcb H^1(0,1))$ be the set of measurable functions $f : [0, T] \to \mcb H^1(0, 1)$ such that
\begin{equation*}
 \int_0^T\|f_t\|^2_{\mcb  H^1(0, 1)} dt<\infty.
\end{equation*}
\end{defin}

From here on, we fix an initial measurable profile $\mathfrak g=(\mathfrak g^A, \mathfrak g^B, \mathfrak g^\emptyset):[0,1]\to[0,1]^3$ with $\mathfrak g^A+ \mathfrak g^B+ \mathfrak g^\emptyset=1$.  This will correspond to the initial condition in all the PDEs that we derive. 
\begin{defin}
\label{def:Dirichlet}
We say that  $\rho=(\rho^A,\rho^B,\rho^\emptyset ):[0,T]\times[0,1] \to [0,1]^3$ is a weak solution of the   system of parabolic equations with Dirichlet boundary conditions
\begin{equation}\label{eq:Dirichlet}
\begin{cases}
&\hspace{0.15cm}\partial_{t}\rho^\alpha=\Delta\, {\rho}^\alpha +\beta\nabla [\rho^\alpha(\rho^{\alpha+1}-\rho^{\alpha+2})],  \\
&{ \rho}^\alpha _{t}(0)=r_\alpha, \quad { \rho}^\alpha_{t}(1)=\tilde r_\alpha,\quad t \in (0,T],\\
& \rho_0^\alpha(u)=\mathfrak g^\alpha(u), \quad u \in [0,1],
\end{cases} \quad \qquad \alpha\in\{A,B,\emptyset\},
\end{equation}
if for each  $\alpha\in\{A,B,\emptyset\}$  it holds:
\begin{itemize}
\item [$\mathbf{D1.}$]$\rho^\alpha \in L^2([0,T]\times \mcb H^1(0,1))$, \item [$\mathbf{D2.}$] ${ \rho} ^\alpha_{t}(0)=r_\alpha$ and ${ \rho}^\alpha_{t}(1)=\tilde r_\alpha$ for a.e. $t \in (0,T]$, \item  [$\mathbf{D3.}$]  for all $t\in [0,T]$ and any $\phi \in C_c^{2} ([0,1])$, $F^\alpha_{\rm Dir}(t,\phi,\rho,\mathfrak{g})=0$, where
\begin{equation}\label{eq:weak_dir}
F^\alpha_{\rm Dir}(t,\phi,\rho,\mathfrak{g}):=\langle \rho^\alpha_{t},  \phi\rangle  -\langle \mathfrak g^\alpha,   \phi \rangle
- \int_0^t\langle \rho^\alpha_{s},\Delta \phi\rangle -\beta\langle \rho^\alpha_s(\rho_s^{\alpha+1}-\rho_s^{\alpha+2}),\nabla  \phi\rangle ds.
\end{equation}
\end{itemize}
\end{defin}

\begin{defin}
\label{def:Robin_4}
Let $\kappa_1, \kappa_2\in\mathbb R$. We say that  $\rho=(\rho^A,\rho^B,\rho^\emptyset ):[0,T]\times[0,1] \to [0,1]^3$ is a weak solution of the system of parabolic equations with Robin boundary conditions
 \begin{equation}\label{eq:Robin_4}
 \begin{cases}
\hspace{0.45cm} \partial_{t}\rho^\alpha=\Delta\, {\rho}^\alpha +\beta\nabla [\rho^\alpha (\rho^{\alpha+1}-\rho^{\alpha+2})]\\
 { \nabla \rho}^\alpha _{t}(0)=-\beta\rho^\alpha_t(0)(\rho^{\alpha+1}_t(0)-\rho^{\alpha+2}_t(0))\\
 \hspace{1.75cm}-\kappa_1\Big((2r_{\alpha+2}-1)\rho_t^\alpha(0)-2r_\alpha\rho_t^{\alpha+1} (0)+r_\alpha\Big)-\kappa_2(r_\alpha-\rho_t^\alpha(0)), \quad t \in (0,T],\\
  { \nabla \rho}^\alpha _{t}(1)=-\beta\rho^\alpha_t(1)(\rho^{\alpha+1}_t(1)-\rho^{\alpha+2}_t(1))\\
  \hspace{1.75cm}-\kappa_1\Big((2\tilde r_{\alpha+2}-1)\rho_t^\alpha(1)-2\tilde r_\alpha\rho_t^{\alpha+1}(1)+\tilde r_\alpha)\Big)+\kappa_2(\tilde r_\alpha-\rho_t^\alpha(1)), \quad t \in (0,T],\\
  \hspace{0.35cm} \rho_0^\alpha(u)  = \mathfrak g^\alpha(u), \quad u \in [0,1],
 \end{cases}
\end{equation}
if for each  $\alpha\in\{A,B,\emptyset\} $  it holds: \begin{itemize}
\item [$\mathbf{R1.}$]$\rho^\alpha \in L^2([0,T]\times \mcb H^1(0,1))$,  
\item  [$\mathbf{R2.}$] for all $t\in [0,T]$ and any $\phi \in C^{1,2} ([0,T]\times[0,1])$, $F^\alpha_{\rm Rob}(t,\phi,\rho,\mathfrak{g})=0$, where
\begin{equation}\label{eq:weak_rob_4}\begin{split}
F^\alpha_{\rm Rob}(t,\phi,\rho,\mathfrak{g}):&=\langle \rho^\alpha_{t},  \phi_{t}\rangle  -\langle \mathfrak g^\alpha,   \phi_{0} \rangle
- \int_0^t\langle \rho^\alpha_{s},(\Delta + \partial_s) \phi_{s}\rangle ds\\&+\beta\int_0^t\langle \rho^\alpha_s(\rho_s^{\alpha+1}-\rho_s^{\alpha+2}),\nabla\phi_s\rangle ds\\&-\int_{0}^t\nabla \phi_s(0)\rho^\alpha_s(0)-\nabla \phi_s(1)\rho^\alpha_s(1)ds\\&-\kappa_1\int_0^t  \phi_s(0){\left[(2r_{\alpha+2}-1)\rho^\alpha_s(0)-2r_\alpha\rho^{\alpha+1}_s(0)+r_\alpha\right] }ds\\
  &-\kappa_1\int_0^t  \phi_s(1){\left[(1-2\tilde r_{\alpha+2})\rho^\alpha_{s}(1)+2\tilde r_\alpha\rho_s^{\alpha+1}(1)-\tilde r_\alpha\right]} ds\\&-\kappa_2\int_0^t \left[\phi_s(0)(r_\alpha-\rho^\alpha_s(0))+\phi_s(1)(\tilde r_\alpha-\rho^{\alpha}_s(1))\right]ds.\end{split}
\end{equation}
\end{itemize}
\end{defin}
\begin{remark}
We observe that in the Dirichlet regime we use test functions which are only space dependent, while in the Robin regime we use test functions also time dependent. The reason for this is that in the Dirichlet case, this space of test functions is enough to prove uniqueness of the weak solution, while in the Robin case we need a bigger space of test functions.
\end{remark}
In Section \ref{sec:uni_dir} we prove the uniqueness of weak solutions of the equation with Dirichlet boundary conditions, as given in Definition \ref{def:Dirichlet}. In Appendix \ref{sec:uni_rob} we also give an heuristic argument for the proof of  the uniqueness of weak solutions  in the Robin case, by assuming that the solution itself can be taken as a test function, which would need later a regularization both in time and space. Nevertheless, we believe that since we did not impose any boundary condition on the test function, our notion of solution should be unique.

\subsection{Hydrodynamic limit }\label{sec:hydrodynamics}

Here we state the main theorem concerning the hydrodynamic limit of the empirical measure.
Assuming the relations \eqref{r_AB=r_B} and \eqref{relations_of_r_AB...}, denote by $\{\eta_t:\,t\geq0\}$ the continuous-time Markov process on $\Omega_N$ with generator $N^2\mcb{L}_N$, and by $\mathbb P_{\mu_N}$ the probability measure that this process induces on the Skorohod space ${\mcb D}([0,T],\Omega_N)$, when it starts from a distribution $\mu_N$ on $\Omega_N$. Denote by $\mathbb{E}_{\mu_N}$ the expectation with respect to~$\mathbb{P}_{\mu_N}$.
Now we fix our initial set of probability measures.
 Fix an initial measurable profile $\mathfrak g=(\mathfrak g^A, \mathfrak g^B, \mathfrak g^\emptyset):[0,1]\to[0,1]^3$ with $\mathfrak g^A+ \mathfrak g^B+ \mathfrak g^\emptyset=1$.
 
\begin{defin}\label{def:measure:associated}
	We say that a sequence  of probability measures {$(\mu_N)_{N \geq 1}$} on $\Omega_N$  is {associated with the profile $\mathfrak g$} if for every continuous  function $\phi$,  for any $\epsilon >0$ and $\alpha\in\{A,B,\emptyset\}$, it holds
	\begin{equation}\label{measure_associated_profile}
	\lim_{N \rightarrow \infty} \mu_N \Big( \eta \in \Omega_N: \big|\dfrac{1}{N}\sum_{x \in \Lambda_{N} }\phi\left(\tfrac{x}{N} \right)\xi_x^\alpha(\eta) - \int_0^1 \phi(u) \mathfrak g^\alpha (u) du \big| > \epsilon \Big) =0. 
	\end{equation}
\end{defin}
\begin{thm}
\label{th:hyd_ssep}
Let  $\mathfrak g=(\mathfrak g^A, \mathfrak g^B, \mathfrak g^\emptyset):[0,1]\to[0,1]^3$ be a measurable function and let $\lbrace\mu _{N}\rbrace_{N\geq 1}$ be a sequence of probability measures  associated with $\mathfrak g$. For any $t\in[0,T]$, for any $\phi\in  C^0([0,1])$,  any $\alpha\in\{A,B,\emptyset\}$, and any $\epsilon>0$ it holds
\begin{equation}\label{limHidreform}
 \lim _{N\to\infty } \mathbb P_{\mu _{N}}\left( \eta_{\cdot} : \Big|\dfrac{1}{N}\sum_{x \in \Lambda_{N} }\phi\left(\tfrac{x}{N} \right)\xi_x^\alpha(\eta_{t}) -\int_0^1 \phi(u) \rho_{t}^\alpha  (u) du \Big|    > \epsilon \right)= 0,
\end{equation}
where $\rho=(\rho^A,\rho^B,\rho^\emptyset ):[0,T]\times[0,1] \to [0,1]^3$ is the unique weak solution of : 
\begin{itemize}
\item [a)]  \eqref{eq:Dirichlet},  if $\delta<1\leq \theta$;
\item [b)]  \eqref{eq:Robin_4} with:
\begin{itemize}
\item [b1)] $\kappa_1=\kappa_2=0$, if $\theta \geq \delta>1$.
 
\item[b2)] $\kappa_1=\tilde \beta/2$ and  $\kappa_2=1$, if $\theta =1=\delta$ and $\tilde\beta<4/3$. 
\item [b3)] $\kappa_1=0$ and $\kappa_2=1$, if $\theta >1$ and $\delta=1$.

\end{itemize}
\end{itemize}
\end{thm}

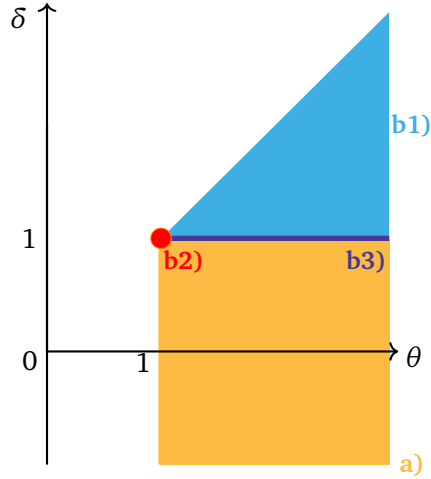
\begin{figure}[htb!]
    \centering
	\begin{tikzpicture}[scale=0.15]
	\node[left, black] at (-15.8,9.5) {$\delta$};
	\node[right, black] at (15.5,-20.5) {$\theta$};
	\node[black] at (-16.5,-20.8) {$0$};
	\node[left] at (-15,-10) {\textcolor{black}{$1$}};
	\node[left] at (-5,-21) {\textcolor{black}{$1$}};

	\fill[color=white] (-15,-10) rectangle (-5,10);
		\fill[color=CornflowerBlue]  (-5,-10) -- (15,-10) -- (15,10)-- cycle;
	\fill[color=white] (-15,-20) -- (-5,-10) -- (-15,-10);
	\fill[color=Dandelion] (-5,-10) rectangle (15,-20);

 	\fill[color=Dandelion] (-5,-30) rectangle (15,-10);

    \node[] at (-3,-12) {\small{{\textbf{\textcolor{red}{b2)}}}}};
    \node[] at (17,0) {\small{{\textbf{\textcolor{CornflowerBlue}{b1)}}}}};
    \node[] at (13,-12) {\small{{\textbf{\textcolor{BlueViolet}{b3)}}}}};
       \node[] at (17,-30) {\small{{\textbf{\textcolor{Dandelion}{a)}}}}};
	
	\draw[-,=latex, BlueViolet,line width=2pt] (-5,-10) -- (15, -10) ;
		\draw[-,=latex,Dandelion,line width=2pt] (-5,-10) -- (-5, -30) ;
	
	
	\draw[<-, thick, black] (-15,10.8) -- (-15,-20);
		\draw[-, thick, black] (-15,0) -- (-15,-30);
	\draw[<-,  thick, black] (15.8,-20) -- (-15,-20);
	
	\draw[orange,fill=red] (-5,-10) circle (9mm);

	\end{tikzpicture}
	\end{figure}
\begin{remark}
		We note that in item b2) in the statement of the theorem we imposed $\tilde\beta<4/3$. Nevertheless, the model is well defined if $\tilde\beta<2$, see \eqref{eq:rest_beta_tilde}. This more restrictive condition comes from the fact that our proof of the uniqueness of weak solutions only works  when $\tilde\beta<4/3$. If the uniqueness of weak solutions can be proved for $\tilde\beta<2$, then our result also holds for those values of $\tilde\beta$.
\end{remark}

\section{Heuristics for the hydrodynamic equation}\label{sec:heuristics}

In this section we give an heuristic argument which is the basis to derive the weak formulation of the respective hydrodynamic equations. Let $\mcb M$ be the space of positive measures on $[0,1]$ with total mass bounded by $1$ and endowed with the weak topology. For each $\alpha\in\{A,B,\emptyset\}$, denote by $\pi^{N,\alpha}:\Omega_N\to{\mcb M}$ the function that associates to each configuration $\eta$ the measure obtained by assigning mass $1/N$ to each particle of type $\alpha$:
\begin{equation*}
	\pi^{N,\alpha}(\eta,du)=\frac{1}{N}\sum_{x\in\Lambda_N}\xi^\alpha_x(\eta)\delta_{\frac xN}(du).
\end{equation*}
Thus, the empirical process $\pi^{N,\alpha}_t:=\pi^{N,\alpha}(\eta_t)$ is a Markov process on the space $\mcb M$. 

For a given function $\phi:[0,T]\times[0,1]\to\mathbb R$ we denote by  $\langle\pi^{N,\alpha}_t,\phi_t\rangle$ the integral of $\phi_t(\cdot)$ with respect to the measure $\pi^{N,\alpha}_t$ and observe that by Dynkin's formula,
\begin{equation}\label{dynkins_formula}
 \mcb M_t^{N,\alpha}(\phi)=\langle\pi^{N,\alpha}_t,\phi_t\rangle -\langle\pi^{N,\alpha}_0,\phi_0\rangle -\int_0^t  (N^2\mcb L_N+\partial_s)\langle\pi^{N,\alpha}_s,\phi_s\rangle ds
\end{equation}
is a martingale with respect to the natural filtration.
A simple, but long, computation shows that  last identity can be written as
\begin{equation}
\begin{aligned}\label{eq:martingaleA}
  \mcb M_t^{N,\alpha}(\phi)=\,&\langle\pi^{N,\alpha}_t,\phi_t\rangle -\langle\pi^{N,\alpha}_0,\phi_0\rangle -\int_0^t  \,\langle\pi^{N,\alpha}_s,\partial_s\phi_s\rangle ds-\int_0^t  \,\langle\pi^{N,\alpha}_s,\Delta_N\phi_s\rangle ds\\&
  +\frac{\beta}{N}\int_0^t  \sum_{x=1}^{N-2}\nabla^+_N\phi_s(\tfrac xN) g^\alpha_{x,x+1}(\eta_s)ds\\
  &-\frac{N^2}{N^{1+\delta}}\int_0^t \left[\phi_s(\tfrac{1}{N})f_1^\alpha(\eta_s)+\phi_s(\tfrac{N-1}{N}){f_{N-1}^\alpha(\eta_s)}\right]ds\\
  &-\int_0^t  \left[\nabla^+_N\phi_s(0)\xi^\alpha_1(\eta_s)-\nabla^+_N\phi_s(\tfrac{N-1}{N})\xi^\alpha_{N-1}(\eta_s)\right]ds\\
  &-\frac{\tilde \beta N^2}{N^{1+\theta}}\int_0^t  \left[\phi_s(\tfrac 1N){h_1^\alpha(\eta_s)}+\phi_s(\tfrac{N-1}{N}){h_{N-1}^\alpha(\eta_s)}\right]ds,
 \end{aligned}
 \end{equation}
 where
 \begin{equation}\label{eq:martingaleA2}
 \begin{aligned}
  g^\alpha_{x,x+1}(\eta) & = \frac{1}{2}\left[\xi^\alpha_{x+1}(\eta)(\xi^{\alpha+1}_{x}(\eta)-\xi^{\alpha+2}_{x}(\eta))+\xi^{\alpha}_{x}(\eta)(\xi^{\alpha+1}_{x+1}(\eta)-\xi^{\alpha+2}_{x+1}(\eta))\right],\\
  f_1^\alpha(\eta) & = (-r_{\alpha+1}-r_{\alpha+2})\xi^\alpha_1(\eta)+r_{\alpha}\xi^{\alpha+1}_1(\eta)+r_{\alpha}\xi^{\alpha+2}_1(\eta),\\
  f_{N-1}^\alpha(\eta) & =  (-\tilde r_{\alpha+1}-\tilde r_{\alpha+2})\xi^\alpha_{N-1}(\eta)+\tilde r_{\alpha}\xi^{\alpha+1}_{N-1}(\eta)+\tilde r_{\alpha}\xi^{\alpha+2}_{N-1}(\eta),\\
  h_1^\alpha(\eta) & = \frac{1}{2}\left[(r_{\alpha+2}-r_{\alpha+1})\xi^\alpha_1(\eta)-r_{\alpha}\xi^{\alpha+1}_1(\eta)+r_{\alpha}\xi^{\alpha+2}_1(\eta)\right], \\
   h_{N-1}^\alpha(\eta) & = \frac{1}{2}\left[(\tilde r_{\alpha+1}-\tilde r_{\alpha+2})\xi^\alpha_{N-1}(\eta)+\tilde r_{\alpha}\xi^{\alpha+1}_{N-1}(\eta)-\tilde r_{\alpha}\xi^{\alpha+2}_{N-1}(\eta)\right]  . 
   \end{aligned}
 \end{equation}
Above we used the usual notions of discrete Laplacian and discrete derivative as: 
\begin{equation*}
 \begin{aligned}
\nabla^+_N\phi(\tfrac{x}{N})&=N[\phi(\tfrac{x+1}{N})-\phi(\tfrac{x}{N})]\quad \textrm{and} \quad
\Delta_N\phi(\tfrac{x}{N})&=N^2[\phi(\tfrac{x+1}{N})-2\phi(\tfrac{x}{N})+\phi(\tfrac{x-1}{N})].
 \end{aligned}
\end{equation*}
Using the relations in \eqref{relations_of_r_AB...} and the exclusion rule $\xi^A_1+\xi^B_1+\xi^\varnothing_1=1$, we have
\begin{equation}\label{f_1f_N-1,h_1,h_N-1}
\begin{aligned}
f_1^\alpha(\eta) & = r_\alpha-\xi^\alpha_1(\eta), \\
f_{N-1}^\alpha(\eta) & =\tilde r_\alpha-\xi^\alpha_{N-1}(\eta),  \\
h_1^\alpha(\eta) & = \frac{1}{2}\left[(2r_{\alpha+2}-1)\xi^\alpha_1(\eta)-2r_\alpha\xi^{\alpha+1}_1(\eta)+r_\alpha\right], \\
h_{N-1}^\alpha(\eta) & = \frac{1}{2}\left[(1-2\tilde r_{\alpha+2})\xi^\alpha_{N-1}(\eta)+2\tilde r_\alpha\xi^{\alpha+1}_{N-1}(\eta)-\tilde r_\alpha\right] . 
\end{aligned}
\end{equation}
At this point we need to analyse each term in \eqref{eq:martingaleA}. In the next subsection, we explain how to obtain heuristically the weak formulation presented in Definition~\ref{def:Dirichlet}. The rigorous proof will be given in Section  \ref{sec:charac}.

\subsection{The Dirichlet case}\label{sec:Dirichlet}

Consider $\phi\in C_c^{2}([0,1])$. Since $\phi$ has a compact support, the third,  fourth and fifth lines of~\eqref{eq:martingaleA} vanish,  for $N$ sufficiently big.  This means that we can rewrite~\eqref{eq:martingaleA} as
\begin{equation}\label{eq:martingaleA_sec}
 \begin{split}
  &\mcb M_t^{N,\alpha}(\phi)=\langle\pi^{N,\alpha}_t,\phi\rangle -\langle\pi^{N,\alpha}_0,\phi\rangle -\int_0^t ds \,\langle\pi^{N,\alpha}_s,\Delta_N\phi\rangle\\
  &{+}\frac{\beta}{2N}\int_0^t\!\! ds \sum_{x=1}^{N-2}\nabla^+_N\phi(\tfrac xN)  \Big[\!\xi^\alpha_{x+1}(\eta_s)(\xi^{\alpha+1}_{x}(\eta_s)-\xi^{\alpha+2}_{x}(\eta_s))+\xi^\alpha_{x}(\eta_s)(\xi^{\alpha+1}_{x+1}(\eta_s)-\xi^{\alpha+2}_{x+1}(\eta_s))\!\Big]
 \end{split}\end{equation}
plus terms that vanish as $N\to+\infty$.
For $\ell\in\N$ and $x\in\Z$ we introduce the averages on boxes of size $\ell$, one to the right of $x$, another  one to the left of $x$:
\begin{equation}\label{eq:average}
\ora{\xi}^{\alpha,\ell}_x(\eta)=\frac{1}{\ell}\sum_{y=x+1}^{x+\ell}\xi^\alpha_y(\eta),\qquad \ola{\xi}^{\alpha,\ell}_x(\eta)=\frac{1}{\ell}\sum_{y=x-\ell}^{x-1}\xi^\alpha_y(\eta).
\end{equation}
First we observe that by paying a price of order $O(\epsilon)$,  we can restrict the sum in the  last line of \eqref{eq:martingaleA_sec} to $x\in\Lambda_N^{\epsilon}$ where  \begin{equation}\label{eq:rest_set}
\Lambda_N^{\epsilon}:=\{\epsilon N+1, \cdots, N-1-\epsilon N\}.\end{equation} Above and in what follows,  $\epsilon N$ represents $\lfloor\epsilon N\rfloor.$ 
Using repeatedly Lemma \ref{eq:global_RL} we can rewrite the last line of \eqref{eq:martingaleA_sec}  as 
 \begin{equation}\begin{split}
 \label{eq:martingaleA_sec_1}
\frac{\beta}{2N}\int_0^t ds \sum_{x\in\Lambda_N^{\epsilon}}\nabla^+_N\phi(\tfrac xN)\Big[\xi^\alpha_{x+1}&(\eta_s)(\ola\xi^{\alpha+1,\epsilon N}_{x}(\eta_s)-\ola \xi^{\alpha+2, \epsilon N}_{x}(\eta_s))\\&+\xi^\alpha_{x}(\eta_s)(\ora \xi^{\alpha+1,\epsilon N}_{x+1}(\eta_s)-\ora\xi^{\alpha+2, \epsilon N}_{x+1}(\eta_s))\Big]
 \end{split}
 \end{equation}
Before we go on  we explain how we do it. First, we split the  integral on the second line of~\eqref{eq:martingaleA_sec}, with the sum restricted to $\Lambda_N^{\epsilon}$  in several terms, i.e. we write it as 
\begin{equation}\label{eq:martingaleA_third}
 \begin{split}
  &\frac{\beta}{2N}\int_0^t ds \sum_{x\in\Lambda_N^\epsilon}\nabla^+_N\phi(\tfrac xN)  \xi^\alpha_{x+1}(\eta_s)\xi^{\alpha+1}_{x}(\eta_s)-\frac{\beta}{2N}\int_0^t ds \sum_{x\in\Lambda_N^\epsilon}\nabla^+_N\phi(\tfrac xN)\xi^\alpha_{x+1}(\eta_s) \xi^{\alpha+2}_{x}(\eta_s) \\
  +&\frac{\beta}{2N}\int_0^t ds \sum_{x\in\Lambda_N^\epsilon}\nabla^+_N\phi(\tfrac xN)\xi^\alpha_{x}(\eta_s)\xi^{\alpha+1}_{x+1}(\eta_s)- \frac{\beta}{2N}\int_0^t ds \sum_{x\in\Lambda_N^\epsilon}\nabla^+_N\phi(\tfrac xN)\xi^\alpha_{x}(\eta_s)\xi^{\alpha+2}_{x+1}(\eta_s).
 \end{split}\end{equation}

Now in the first integral in last display,  we apply Lemma \ref{eq:global_RL}  with $G_s(\tfrac{x}{N})=\frac{\beta}{2}\nabla^+_N \phi_s(\tfrac xN)$, $\tau_x\psi(\eta)=\xi^{\alpha}_{x+1}(\eta)$ and we make the replacement of $\xi^{\alpha+1}_x$ by $\ola\xi^{\alpha+1, \epsilon N}_x$. 
In all the remaining terms we do a similar replacement. 
Noting that for $u\in[0,1]$ if 
 \begin{equation}\label{iotas}
\ora{\iota_\epsilon}(\tfrac xN)(u)=\frac{1}{\epsilon}\textbf{1}_{\left(\tfrac xN, \tfrac xN+\epsilon\right]}(u)\quad \textrm{and}\quad\ola{\iota_\epsilon}(\tfrac xN)(u)=\frac{1}{\epsilon}\textbf{1}_{\left[ \tfrac xN-\epsilon, \tfrac xN\right)}(u),\end{equation} then, for any $s\in[0,T]$
\begin{equation}\label{eq:ident}
\langle \pi_s^{N,\alpha}, \ora{\iota_\epsilon}(\tfrac xN)\rangle =\frac{1}{\epsilon N}\sum_{z=x+1}^{x+\epsilon N}\xi^\alpha_z(\eta_s)\,,\quad\quad \langle \pi_s^{N,\alpha}, \ola{\iota_\epsilon}(\tfrac xN)\rangle =\frac{1}{\epsilon N}\sum_{z=x-\epsilon N}^{x-1}\xi^\alpha_z(\eta_s),\end{equation}
and we can rewrite \eqref{eq:martingaleA_third} as 
 \begin{equation}\begin{aligned}\label{eq:martingaleA_thi}
\frac{\beta}{2N}\int_0^t ds \sum_{x\in\Lambda_N^{\epsilon }}\nabla^+_N\phi(\tfrac xN)  \Big[\xi^\alpha_{x+1}(\eta_s)&\Big(\langle \pi_s^{N,\alpha+1}, \ola\iota_\epsilon(\tfrac xN )\rangle-\langle \pi_s^{N,\alpha+2}, \ola\iota_\epsilon(\tfrac xN)\rangle\Big)\\&+\xi^\alpha_{x}(\eta_s)\Big(\langle \pi_s^{N,\alpha+1}, \ora \iota_\epsilon(\tfrac xN)\rangle-\langle \pi_s^{N,\alpha+2}, \ora\iota_\epsilon(\tfrac xN)\rangle\Big)\Big].
 \end{aligned}\end{equation}
By paying a price of order $O(\epsilon)$ in order to convert the sum in $\Lambda_N^\epsilon$ to the whole sum, we can  rewrite the last term as
{\begin{equation}\begin{aligned}
\frac{\beta}{2}\int_0^t \Big  \langle \pi_s^{N,\alpha}, \nabla^+_N\phi(\cdot) \Big(\langle \pi_s^{N,\alpha+1}, &\ola{\iota_\epsilon}(\cdot)\rangle-\langle \pi_s^{N,\alpha+2},\ola{\iota_\epsilon}(\cdot)\rangle\Big)\Big\rangle\\&+\Big\langle \pi_s^{N,\alpha}, \nabla^+_N\phi(\cdot) \Big(\langle \pi_s^{N,\alpha+1}, \ora{\iota_\epsilon}(\cdot)\rangle-\langle \pi_s^{N,\alpha+2},\ora{\iota_\epsilon}(\cdot)\rangle\Big) \Big\rangle ds.
\end{aligned}\end{equation}}
Putting last identity back to \eqref{eq:martingaleA_sec} we have that 
{\begin{equation*}
 \begin{split}\label{eq:martingaleA_sec_last}
  \mcb M_t^{N,\alpha}(\phi)=\langle\pi^{N,\alpha}_t,\phi\rangle -&\langle\pi^{N,\alpha}_0,\phi\rangle -\int_0^t ds\langle\pi^{N,\alpha}_s, \,\Delta_N\phi\rangle\\
&+\frac{\beta}{2}\int_0^t  \Big\langle \pi_s^{N,\alpha}, \nabla^+_N\phi(\cdot) \Big(\langle \pi_s^{N,\alpha+1},\ola{\iota_\epsilon}(\cdot)\rangle-\langle \pi_s^{N,\alpha+2},\ola{\iota_\epsilon}(\cdot)\rangle\Big)\Big\rangle\\&\quad \quad \quad \quad \quad +\Big\langle \pi_s^{N,\alpha}, \nabla^+_N\phi(\cdot) \Big(\langle \pi_s^{N,\alpha+1}, \ora{\iota_\epsilon}(\cdot)\rangle-\langle \pi_s^{N,\alpha+2},\ora{\iota_\epsilon}(\cdot)\rangle\Big) \Big\rangle ds,
 \end{split}\end{equation*}}
 plus terms that vanish as $N\to+\infty$ and $\epsilon\to 0$.
 From the computations of Section  \ref{sec:tightness} we will see that the martingale vanishes in $L^2(\mathbb P_{\mu_N})$, as $N\to+\infty$. Moreover,  from the results of Section~\ref{sec:tightness}, we can also assume the weak convergence of $\pi^{N,\alpha}_t(\eta,du)$ to $\pi^\alpha_t(du)=\rho^\alpha_t(u)du$, for each $\alpha\in\{A,B,\emptyset\}$. On the other hand, since $\rho^\alpha_t(u)\in[0,1]$ for all $t\in[0,T]$ and $u\in[0,1]$, from Lebesgue's differentiation theorem it holds
\begin{equation}\label{eq:passage_rho}
\lim_{\epsilon\to 0}\Big|\rho^\alpha_s(u)-\tfrac{1}{\epsilon}\int_{u}^{u+\epsilon} \rho^\alpha_s(v)dv\Big| =0\quad \textrm{and }\quad \lim_{\epsilon\to 0}\Big|\rho^\alpha_s(u)-\tfrac{1}{\epsilon}\int_{u-\epsilon}^{u} \rho^\alpha_s(v)dv\Big| =0
\end{equation}
for almost every $u\in[0,1]$. 
 As a consequence, \eqref{eq:martingaleA_sec_last} converges, as $N\to+\infty$, to 
  \begin{align}\label{eq:martingaleA_limit_dir}
0&=\langle\rho^{\alpha}_t,\phi_t\rangle -\langle\rho^\alpha_0,\phi_0\rangle -\int_0^t ds\langle\rho^{\alpha}_s, \,(\partial_s+\Delta)\phi_s\rangle+\beta\int_0^t ds  \langle \rho_s^{\alpha}  ( \rho_s^{\alpha+1} -\rho_s^{\alpha+2}), \nabla\phi_s\rangle.
 \end{align}

We note that above we gave an heuristic argument to  obtain the integral formulation as given in \eqref{eq:weak_dir}. We note that  the proof of condition $\mathbf{D2}.$ is explained in Section \ref{section_properties}. 
  
 In the next subsection, we explain a similar argument as given above, to obtain the weak formulation of \eqref{eq:Robin_4}.

\subsection{The Robin cases}\label{sec:RobinNeumann}

Now we do not assume any condition on the test function, i.e. we consider $\phi\in C^{1,2}([0,T]\times [0,1])$. In this case we need to analyse carefully each line in \eqref{eq:martingaleA}.  The analysis of the first and second lines can be done exactly as in the Dirichlet case of the previous subsection. Now let us analyse the remaining terms.  The fourth line does not depend on the value of the parameters $\theta$ nor $\delta$. From Lemma \ref{lem:local_RL} we can rewrite that term as 
 \begin{align}\label{eq:third_term}
 -\int_0^t ds \left[\nabla^+_N\phi_s(0)\ora\xi^{\alpha,\epsilon N}_1(\eta_s)-\nabla^+_N\phi_s(\tfrac{N-1}{N})\ola\xi^{\alpha,\epsilon N}_{N-1}(\eta_s)\right],
 \end{align}
plus terms that vanish as $N\to+\infty$ and $\epsilon\to0$. As in the previous case, we could now try to use 
\eqref{eq:passage_rho} but as we have seen above, it is only true for almost every $u\in[0,1]$ and here we need this to be true for $u=0$ and $u=1$. The way to conclude that result,  is to use   Lemma \ref{lem:boundary_terms}, which is similar to Lemma 6.2 in \cite{BDGN}. 
From those results we conclude that the last display can be written as 
\begin{align}\label{eq:third_term_1}
 -\int_0^t  \left[\nabla\phi_s(0)\rho^\alpha_s(0)-\nabla\phi_s(1)\rho^{\alpha}_{s}(1)\right]ds
 \end{align}
plus terms that vanish as $N\to+\infty.$
Now we analyse the third line of \eqref{eq:martingaleA}. If $\delta>1$, it is simple to conclude that the whole term vanishes as $N\to+\infty$. 
Let us now see the case $\delta=1$. Then it rewrites as 
\begin{align*}
-\int_0^t \big[\phi_s(\tfrac{1}{N})(r_\alpha-\xi^\alpha_1(\eta_s))+\phi_s(\tfrac{N-1}{N})(\tilde r_\alpha-\xi^\alpha_{N-1}(\eta_s))\big]ds.
 \end{align*}
 Applying Lemma \ref{lem:local_RL} we can rewrite it as 
 \begin{align*}
-\int_0^t ds\left[\phi_s(\tfrac{1}{N})(r_\alpha-\ora\xi^{\alpha,\epsilon N}_{1}(\eta_s))+\phi_s(\tfrac{N-1}{N})(\tilde r_\alpha-\ola\xi^{\alpha,\epsilon N}_{N-1}(\eta_s))\right]
 \end{align*}
 plus terms that vanish as $N\to+\infty$ and $\epsilon\to0$. As before,  
 using the same  arguments as the ones we used in the analysis of the  previous terms, we can rewrite last display as 
  \begin{align*}
-\int_0^t \left[\phi_s(0)(r_\alpha-\rho^\alpha_s(0)+\phi_s(1)(\tilde r_\alpha-\rho^{\alpha}_s(1)\right]ds.
 \end{align*}
 Finally we treat the last term in \eqref{eq:martingaleA}. For $\theta>1$ it is simple to see that term vanishes as $N\to+\infty$. For $\theta=1$ and repeating the same arguments as before, that term can be rewritten as 
 \begin{align*}
 &-\frac{\tilde \beta}{2}\int_0^t \phi_s(0)\left[(2r_{\alpha+2}-1)\rho^\alpha_s(0)-2r_\alpha\rho^{\alpha+1}_s(0)+r_\alpha\right] ds\\
  &-\frac{\tilde \beta}{2}\int_0^t  \phi_s(1)\left[(1-2\tilde r_{\alpha+2})\rho^\alpha_{s}(1)+2\tilde r_\alpha\rho_s^{\alpha+1}(1)-\tilde r_\alpha\right] ds.
 \end{align*}
Putting together last results we deduce that  the limit of \eqref{eq:martingaleA} gives
for  $\theta \geq \delta>1$
 \begin{equation}\begin{split}\label{eq:martingaleA_limit_rob_1}
0=\langle\rho^{\alpha}_t,\phi_t\rangle -\langle\rho_0,\phi_0\rangle -\int_0^t ds\langle\rho^{\alpha}_s, \,(\partial_s+\Delta)\phi_s\rangle&+\beta\int_0^t ds  \langle \rho_s^{\alpha}  ( \rho_s^{\alpha+1} -\rho_s^{\emptyset}), \nabla\phi_s\rangle\\& -\int_0^t  [\nabla\phi_s(0)\rho^\alpha_s(0)-\nabla\phi_s(1)\rho^{\alpha}_{1}]ds,
 \end{split}\end{equation}
 which coincides with \eqref{eq:weak_rob_4}
 for the choice  $\kappa_1=\kappa_2=0$.

For  $\theta >\delta=1$, the limit of \eqref{eq:martingaleA} gives
 \begin{equation}\begin{split}\label{eq:martingaleA_limit_rob_3}
0=\langle\rho^{\alpha}_t,\phi_t\rangle -\langle\rho_0^\alpha,\phi_0\rangle &-\int_0^t ds\langle\rho^{\alpha}_s, \,(\partial_s+\Delta)\phi_s\rangle+\beta\int_0^t ds  \langle \rho_s^{\alpha}  ( \rho_s^{\alpha+1} -\rho_s^{\alpha+2}), \nabla\phi_s\rangle\\& -\int_0^t  [\nabla\phi_s(0)\rho^\alpha_s(0)-\nabla\phi_s(1)\rho^{\alpha}_{1}]ds\\ &-\frac{\tilde \beta}{2}\int_0^t  [\phi_s(0){\left[(2r_{\alpha+2}-1)\rho^\alpha_s(0)-2r_\alpha\rho^{\alpha+1}_s(0)+r_\alpha\right] }ds\\
  &-\frac{\tilde \beta}{2}\int_0^t  \phi_s(1){\left[(1-2\tilde r_{\alpha+2})\rho^\alpha_{s}(1)+2\tilde r_\alpha\rho_s^{\alpha+1}(1)-\tilde r_\alpha\right]} ds
 \end{split}\end{equation}
which coincides with \eqref{eq:weak_rob_4} for the choice $\kappa_1=0$ and $\kappa_2=1$.

Finally,  for $\theta=1=\delta$,   the limit of \eqref{eq:martingaleA} gives
  \begin{equation}\begin{split}\label{eq:martingaleA_limit_rob_3}
0=\langle\rho^{\alpha}_t,\phi_t\rangle -\langle\rho_0^\alpha,\phi_0\rangle &-\int_0^t ds\langle\rho^{\alpha}_s, \,(\partial_s+\Delta)\phi_s\rangle+\beta\int_0^t ds  \langle \rho_s^{\alpha}  ( \rho_s^{\alpha+1} -\rho_s^{\alpha+2}), \nabla\phi_s\rangle\\& -\int_0^t  [\nabla\phi_s(0)\rho^\alpha_s(0)-\nabla\phi_s(1)\rho^{\alpha}_{1}]ds\\ &-\frac{\tilde \beta}{2}\int_0^t  [\phi_s(0){\left[(2r_{\alpha+2}-1)\rho^\alpha_s(0)-2r_\alpha\rho^{\alpha+1}_s(0)+r_\alpha\right] }ds\\
  &-\frac{\tilde \beta}{2}\int_0^t  \phi_s(1){\left[(1-2\tilde r_{\alpha+2})\rho^\alpha_{s}(1)+2\tilde r_\alpha\rho_s^{\alpha+1}(1)-\tilde r_\alpha\right]} ds\\
  &-\int_0^t [\phi_s(0)(r_\alpha-\rho^\alpha_s(0)+\phi_s(1)(\tilde r_\alpha-\rho^{\alpha}_s(1)]ds,
 \end{split}\end{equation}
 which coincides with \eqref{eq:weak_rob_4}
for the choice $\kappa_1=\tilde \beta/2$ and $\kappa_2=1$.

\section{Tightness}\label{sec:tightness}

For $\alpha\in\{A,B,\emptyset\}$ denote by $\mathbb{Q}^{N,\alpha}=\mathbb P_{\mu_N}(\pi^{N,\alpha})^{-1}$ the probability measure  on ${\cal D}([0,T],\mcb M)$ induced by the Markov process $\{\pi^{N,\alpha}_t:t\geq0\}$ and by the initial measure $\mu_N$. In this section we will prove that the sequence $\{\mathbb{Q}^{N,\alpha}\}_{N\geq1}$ is tight and this will ensure that every subsequence of $\{\mathbb{Q}^{N,\alpha}\}_{N\geq1}$ has a further subsequence which is weakly convergent.
\begin{prop} \label{prop:tight}For $\alpha \in\{A,B,\emptyset\}$, the sequence of measures $\{\mathbb{Q}^{N,\alpha}\}_{N\geq1}$ is tight.
\end{prop}
\begin{proof}By Proposition~1.7 of Chapter 4 of \cite{kipnisLandim} it is enough to show, for every $\phi\in C^1([0,1])$, that the sequence of measures associated with the real-valued process $\{\langle\pi^{N,\alpha}_t,\phi\rangle\}_{N\geq1}$ is tight, and for that it is enough to show that, for every $\phi\in C([0,1])$ and $\varepsilon>0$,
\begin{equation}\label{tightness1}
\lim_{\gamma\to0}\limsup_{N\to\infty}\mathbb P_{\mu_N}\left[\sup_{|t-s|\leq\gamma}|\langle\pi^{N,\alpha}_t,\phi\rangle-\langle\pi^{N,\alpha}_s,\phi\rangle|>\varepsilon\right]=0.
\end{equation}
We begin by showing that \eqref{tightness1} holds for every $\phi\in C^2_c([0,1])$, and we note that the extension to $\phi\in C([0,1])$ is a simple argument that can be seen, for example, in Section 4 of \cite{BMNS17}. Considering the Dynkin's martingale \eqref{dynkins_formula}, in order to conclude \eqref{tightness1} it is enough to prove that
\begin{equation}\label{tightness3}
\lim_{\gamma\to0}\limsup_{N\to\infty}\mathbb E_{\mu_N}\left[\sup_{|t-s|\leq\gamma}\left|\int_s^t N^2 \mcb{L}_N\langle\pi^{N,\alpha}_r,\phi\rangle dr \right|\right]=0
\end{equation}
and
\begin{equation}\label{tightness4}
\lim_{\gamma\to0}\limsup_{N\to\infty}\mathbb E_{\mu_N}\left[\sup_{|t-s|\leq\gamma}|\mcb M^{N,\alpha}_t(\phi)-\mcb M^{N,\alpha}_s(\phi)|\right]=0.
\end{equation}
Since we are assuming that $\phi$ has compact support in $(0,1)$, there exist $N_0$ such that, for all $N\geq N_0$, $\phi(0)=\phi(\tfrac{1}{N})=\phi(\tfrac{N-1}{N})=\phi(1)=0$, and so in \eqref{tightness3} we can replace the full generator $\mcb{L}_N$ by the bulk generator $\mcb{L}^B_N$. Recall the function $g^\alpha_{x,x+1}$ defined in \eqref{eq:martingaleA2}. Since $\xi^\alpha_x(\eta)\in\{0,1\}$ and $|g^\alpha_{x,x+1}(\eta)|\leq1$, by the mean value theorem, for any $r\in[0,T]$, it holds
\begin{align*}
\left|N^2 \mcb{L}_N\langle\pi^{N,\alpha}_r,\phi\rangle\right| &=\left|\frac{1}{N}\sum_{x=1}^{N-1}\Delta_N\phi(\tfrac{x}{N})\xi^\alpha_x(\eta_r)-\frac{\beta}{N}\sum_{x=1}^{N-2}\nabla_N^+\phi(\tfrac{x}{N})g^\alpha_{x,x+1}(\eta_r)\right| \\
& \leq 2\|\phi''\|_\infty+\beta \|\phi'\|_\infty.
\end{align*}
Therefore, there exists a constant $C$ such that $\left|N^2 \mcb{L}_N\langle\pi^{N,\alpha}_r,\phi\rangle\right|<C$ for all $N\geq N_0$, and this implies \eqref{tightness3}. Now let us verify \eqref{tightness4}. By the triangular, Jensen and Doob's inequalities
\begin{equation}\label{tightness5}
\mathbb E_{\mu_N}\left[\sup_{|t-s|\leq\gamma}|\mcb M^{N,\alpha}_t(\phi)-\mcb M^{N,\alpha}_s(\phi)|\right]\leq 4 \mathbb E_{\mu_N}\left[(\mcb M^{N,\alpha}_T(\phi))^2\right]^{1/2}.
\end{equation}
To estimate \eqref{tightness5}, observe that
\begin{equation*}
\mcb M^{N,\alpha}_t(\phi)^2-\int_0^tN^2\mcb{L}_N\langle\pi^{N,\alpha}_s,\phi\rangle^2-2N^2\langle\pi^{N,\alpha}_s,\phi\rangle\mcb{L}_N\langle\pi^{N,\alpha}_s,\phi\rangle ds
\end{equation*}
is a mean zero martingale and a  standard computation shows that the rightmost term in last display is, for all $N\geq N_0$, equal to
\begin{equation}\label{eq:tight2}
\begin{aligned}
\int_{0}^t\frac{1}{N^2}\sum_{x=1}^{N-2}c_{x,x+1}(\eta_s)(\nabla_N\phi(\tfrac{x}{N}))^2[\xi^\alpha_{x+1}(\eta_s)-\xi^\alpha_x(\eta_s)]^2ds.
\end{aligned}
\end{equation}
From this  and the fact that the rates $c_{x,x+1}$ are bounded, we get
\begin{equation*}
\mathbb E_{\mu_N}\left[(\mcb M^{N,\alpha}_T(\phi))^2\right]\leq \frac{C}{N}\|(\phi')^2\|_\infty
\end{equation*}
 which vanishes as $N\to\infty$. This concludes the proof. 
\end{proof}

\section{Characterization of limit points} \label{sec:charac}
Let $\pi^N=(\pi^{N,A},\pi^{N,B},\pi^{N,\emptyset}):\Omega^N\to\mcb M^3$, and consider $\mathbb Q^N =\mathbb P_{\mu_N}(\pi^N)^{-1}$, the probability measure on $\mcb D ([0,T],\mcb M^3)$ induced by the process $\{\pi_t^N: t\geq0\}$ and by the initial measure $\mu_N$. In this section we prove that the limit points $\mathbb Q$ of the sequence $\{\mathbb Q^{N}\}_{N\geq1}$ are concentrated on trajectories, whose marginals are absolutely continuous with respect to the Lebesgue measure and whose densities $\rho^\alpha_t(u)$, $\alpha\in\{A,B,\emptyset\}$, are weak solutions of the corresponding hydrodynamic equation.

Recall the definitions of $F_{\rm Dir}$ and $F_{\rm Rob}$ in \eqref{eq:weak_dir} and \eqref{eq:weak_rob_4}. Let $C_{\rm Dir}=C_c^{1,2}([0,T]\times[0,1])$ and $C_{\rm Rob}=C^{1,2}([0,T]\times[0,1])$.
\begin{prop}Let $\mathbb Q$ be any limit point of the sequence $\{\mathbb Q^{N}\}_{N\geq1}$. Then, for $\alpha\in\{A,B,\emptyset\}$ 
	\begin{equation*}
	\mathbb Q^{N}(\pi_\cdot\in\mcb{D}([0,T],\mcb M^3):\, F^\alpha_{\Gamma}(t,\phi,\rho,\mathfrak{g})=0,\,\forall t\in[0,T],\,\forall\phi\in C_\Gamma)=1,
	\end{equation*}
	where $\Gamma={\rm Dir}$ or $\Gamma={\rm Rob}$, depending on the values of $\delta$ and $\theta$, as in the cases a) and b) in Theorem~\ref{th:hyd_ssep}.
\end{prop}
\begin{proof} We start with the case $\Gamma=\textrm{Dir}$, i.e. the Dirichlet boundary conditions, the proof of the Robin case being almost identical.  It is enough to show that, for any $\phi\in C_c^{1,2}([0,T]\times[0,1])$ and for any $\tilde \delta>0$,
\begin{equation*}
  \mathbb Q(\pi_\cdot\in\mcb{D}([0,T],\mcb M^3):\, \sup_{0\leq t\leq T}|F^\alpha_{{\rm Dir}}(t,\phi,\rho,\mathfrak{g})|>\tilde \delta)=0,
\end{equation*}
Using the definitions~\eqref{eq:weak_dir} and~\eqref{eq:average} and by summing and subtracting appropriate terms we can bound the probability from above by the sum of the following probabilities:
\begin{equation}\label{eq:CLP1}
  \begin{aligned}
 & \mathbb Q\Bigg(\sup_{0\leq t\leq T}\Big|\langle \rho^\alpha_t,\phi_t\rangle -\langle \rho_0^\alpha,\phi_0\rangle -\int_0^t \langle \rho^\alpha_s,\partial_s\phi_s\rangle ds -\int_0^t\int_0^{1-\epsilon}\langle \pi^{\alpha}_s,\ora{\iota_\epsilon}(u)\rangle \Delta\phi_s(u) du ds\\
  &+\frac{\beta}{2}\int_0^t \int_\epsilon^{1-\epsilon}\nabla\phi_s(u) \langle \pi_s^{\alpha},\ora{\iota_\epsilon}(u)\rangle \Big(\langle \pi_s^{\alpha+1}, \ola{\iota_\epsilon}(u)\rangle-\langle \pi_s^{\alpha+2},\ola{\iota_\epsilon}(u)\rangle\Big) du ds
 \\&+\frac{\beta}{2}\!\int_0^t\!\!\!\int_\epsilon^{1-\epsilon}\nabla\phi_s(u) \langle \pi_s^{\alpha},\ola{\iota_\epsilon}(u)\rangle \Big(\langle \pi_s^{\alpha+1}, \ora{\iota_\epsilon}(u)\rangle-\langle \pi_s^{\alpha+2},\ora{\iota_\epsilon}(u)\rangle\Big) du ds
  \Big|>\frac{\tilde \delta}{5}\Bigg),
  \end{aligned}
 \end{equation}
\begin{equation}\label{eq:CLP2}
  \mathbb Q\Bigg(\Big|\langle \rho^\alpha_0,\phi_0\rangle -\langle \rho^\alpha_0,\mathfrak{g}^\alpha\rangle\Big|>\frac{\tilde\delta}{5}\Bigg)
 \end{equation}
 \begin{equation}\label{eq:CLP3}
  \begin{aligned}
\!\!\!  \mathbb Q\Bigg(\sup_{0\leq t\leq T}\!\Big|\frac{\beta}{2}\!\int_0^t\!\!\!\int_\epsilon^{1-\epsilon}\nabla\phi_s(u) &\Big(\rho^\alpha_s(u)(\rho^{\alpha+1}_s(u)-\rho^{\alpha+2}_s(u))\\&-\langle \pi_s^{\alpha},\ora{\iota_\epsilon}(u)\rangle \Big(\langle \pi_s^{\alpha+1}, \ola{\iota_\epsilon}(u)\rangle-\langle \pi_s^{\alpha+2},\ola{\iota_\epsilon}(u)\rangle\Big) du ds\Big|\!\!>\!\frac{\tilde\delta}{5}\Bigg),
  \end{aligned}
  \end{equation}
  \begin{equation}\label{eq:CLP4}
  \begin{aligned}
\!\!\!  \mathbb Q\Bigg(\sup_{0\leq t\leq T}\!\Big|\frac{\beta}{2}\!\int_0^t\!\!\!\int_\epsilon^{1-\epsilon}\nabla\phi_s(u) &\Big(\rho^\alpha_s(u)(\rho^{\alpha+1}_s(u)-\rho^{\alpha+2}_s(u))\\&-\langle \pi_s^{\alpha},\ola{\iota_\epsilon}(u)\rangle \Big(\langle \pi_s^{\alpha+1}, \ora{\iota_\epsilon}(u)\rangle-\langle \pi_s^{\alpha+2},\ora{\iota_\epsilon}(u)\rangle\Big) du ds\Big|\!\!>\!\frac{\tilde\delta}{5}\Bigg),
  \end{aligned}
  \end{equation}
 \begin{equation}\label{eq:CLP5}
  \begin{aligned}
 \mathbb Q\Bigg(\sup_{0\leq t\leq T}\Big|&\int_0^t\langle \rho^\alpha_s,\Delta\phi_s\rangle ds -\int_0^t\int_0^{1-\epsilon}\langle\pi^{\alpha}_s, \ora{\iota_\epsilon}(u)\rangle\Delta\phi_s(u) duds\Big|>\frac{\tilde\delta}{5}\Bigg).
  \end{aligned}
 \end{equation}
Above we used the notation $\ora{\iota_\epsilon}(u)(v)=\frac{1}{\varepsilon}\textbf{1}_{(u,u+\varepsilon]}(v)$ and $\ola{\iota_\epsilon}(u)(v)=\frac{1}{\varepsilon}\textbf{1}_{[u-\varepsilon,u)}(v)$.
Since $\mu_N$ satisfies~\eqref{measure_associated_profile},~\eqref{eq:CLP2} is equal to zero. For~\eqref{eq:CLP5} we use~\eqref{eq:passage_rho} and we  conclude that it goes to zero as $\epsilon\to0$.
To show that~\eqref{eq:CLP3} and~\eqref{eq:CLP4} go to zero, we repeat exactly the previous argument together with \eqref{eq:passage_rho}. 

 Now we have to deal with~\eqref{eq:CLP1}. Since the set inside the probability is an open set, by Portmanteau's theorem\footnote{To use Portmanteau's theorem we actually first have to approximate $\ora{\iota_\epsilon}(u)$ and $\ola{\iota_\epsilon}(u)$ by continuous functions, in such a way that the error vanishes as $\epsilon\to0$.} we can bound~\eqref{eq:CLP1} by
 \begin{equation}\label{eq:CLP11}
  \begin{aligned}
&   \liminf_{N\to\infty} \mathbb Q^{N}\Bigg(\sup_{0\leq t\leq T}\Big|\langle \rho^\alpha_t,\phi_t\rangle -\langle \rho^\alpha_0,\phi_0\rangle -\int_0^t \langle \rho^\alpha_s,\partial_s\phi_s\rangle ds -\int_0^t\int_0^{1-\epsilon}\langle \pi^{N,\alpha}_s,\ora{\iota_\epsilon}(u)\rangle \Delta\phi_s (u)du ds\\
  &+\frac{\beta}{2}\int_0^t \int_\epsilon^{1-\epsilon}\nabla\phi_s(u) \langle \pi_s^{\alpha},\ora{\iota_\epsilon}(u)\rangle \Big(\langle \pi_s^{\alpha+1}, \ola{\iota_\epsilon}(u)\rangle-\langle \pi_s^{\alpha+2},\ola{\iota_\epsilon}(u)\rangle\Big) du ds\\&+\frac{\beta}{2}\!\int_0^t\!\!\!\int_\epsilon^{1-\epsilon}\nabla\phi_s(u) \langle \pi_s^{\alpha},\ola{\iota_\epsilon}(u)\rangle \Big(\langle \pi_s^{\alpha+1}, \ora{\iota_\epsilon}(u)\rangle-\langle \pi_s^{\alpha+2},\ora{\iota_\epsilon}(u)\rangle\Big) du ds  \Big|>\frac{\tilde\delta}{5}\Bigg). 
  \end{aligned}
 \end{equation}
We sum and subtract inside the absolute value the term $\int_0^t N^2{\mcb L_N}\langle\pi^{N,\alpha}_s,\phi_s\rangle ds$, and by~\eqref{eq:martingaleA_sec} we can bound the 	last probability from above by
\begin{equation*}
  \begin{aligned}
&  \mathbb P_{\mu_N}\Bigg(\sup_{0\leq t\leq T}|{\mcb M}_t^N(\phi)|>\frac{\delta}{10}\Bigg)+  \mathbb P_{\mu_N}\Bigg(\sup_{0\leq t\leq T}\Bigg|\int_0^t N^2{\mcb L_N}\langle\pi^{N,\alpha}_s,\phi_s\rangle ds\\&-\int_0^t\int_0^{1-\epsilon}\langle \pi^{N,\alpha}_s,\ora{\iota_\epsilon}(u)\rangle \Delta\phi_s(u) du ds \\
 &+\frac{\beta}{2}\int_0^t \int_\epsilon^{1-\epsilon}\nabla\phi_s(u) \langle \pi_s^{N,\alpha},\ora{\iota_\epsilon}(u)\rangle \Big(\langle \pi_s^{N,\alpha+1}, \ola{\iota_\epsilon}(u)\rangle-\langle \pi_s^{N,\alpha+2},\ola{\iota_\epsilon}(u)\rangle\Big) du ds\\&+\frac{\beta}{2}\!\int_0^t\!\!\!\int_\epsilon^{1-\epsilon}\nabla\phi_s(u) \langle \pi_s^{N,\alpha},\ola{\iota_\epsilon}(u)\rangle \Big(\langle \pi_s^{N,\alpha+1}, \ora{\iota_\epsilon}(u)\rangle-\langle \pi_s^{N,\alpha+2},\ora{\iota_\epsilon}(u)\rangle\Big) du ds  \Bigg|>\frac{\delta}{10}\Bigg). 
  \end{aligned}
  \end{equation*}
From Doob's inequality and~\eqref{eq:tight2}, the first probability vanishes as $N\to\infty$. By making explicit the action of the generator, we can bound the second probability by the sum of the following terms:
\begin{equation}\label{eq:CLP8}
 \mathbb P_{\mu_N}\left(\sup_{0\leq t\leq T}\left|\int_0^t\langle \pi^{N,\alpha}_s,\Delta_N\phi_s\rangle ds-\int_0^t\int_\epsilon^{1-\epsilon}\langle \pi^{N,\alpha}_s,\ora{\iota_\epsilon}(u)\rangle \Delta\phi_s(u) du ds\right|>\frac{\delta}{30} \right),
\end{equation}
\begin{equation}\label{eq:CLP9}
  \begin{aligned}
& \mathbb P_{\mu_N}\Bigg(\sup_{0\leq t\leq T}\Big|\frac{\beta}{2}\int_0^t  \Bigg\{\frac{1}{N}\sum_{x=1}^{N-2}\nabla^+_N\phi_s(\tfrac xN) \xi^\alpha_{x+1}(\eta_s)(\xi^{\alpha+1}_{x}(\eta_s)-\xi^{\alpha+2}_{x}(\eta_s))\Bigg\}\\
&
\qquad-  \int_\epsilon^{1-\epsilon}\nabla\phi_s(u) \langle \pi_s^{N,\alpha},\ora{\iota_\epsilon}(u)\rangle \Big(\langle \pi_s^{N,\alpha+1}, \ola{\iota_\epsilon}(u)\rangle-\langle \pi_s^{N,\alpha+2},\ola{\iota_\epsilon}(u)\rangle\Big) du\Bigg\}ds\Big|>\frac{\delta}{30}\Bigg),
\\
& \mathbb P_{\mu_N}\Bigg(\sup_{0\leq t\leq T}\Big|\frac{\beta}{2}\int_0^t \Bigg\{\frac{1}{N}\sum_{x=1}^{N-2}\nabla^+_N\phi_s(\tfrac xN) \xi^\alpha_{x}(\eta_s)(\xi^{\alpha+1}_{x+1}(\eta_s)-\xi^{\alpha+2}_{x+1}(\eta_s))\Bigg\}\\
&
\qquad-  \int_\epsilon^{1-\epsilon}\nabla\phi_s(u) \langle \pi_s^{N,\alpha},\ola{\iota_\epsilon}(u)\rangle \Big(\langle \pi_s^{N,\alpha+1}, \ora{\iota_\epsilon}(u)\rangle-\langle \pi_s^{N,\alpha+2},\ora{\iota_\epsilon}(u)\rangle\Big) du\Bigg\} ds\Big|>\frac{\delta}{30}\Bigg),
  \end{aligned}
  \end{equation}
To treat \eqref{eq:CLP8} it is enough to use  Taylor expansion of $\phi_s$ and we see that  \eqref{eq:CLP8} vanishes as $N\to\infty$ and $\epsilon\to0$.
 For~\eqref{eq:CLP9}, we can consider only the first probability, since the second one is analogous. First, we split the term $\xi^\alpha_{x}(\eta_s)(\xi^{\alpha+1}_{x+1}(\eta_s)-\xi^{\alpha+2}_{x+1}(\eta_s))$
 and then we compare $\frac{1}{N}\sum_{x}\nabla^+_N\phi_s(\tfrac xN)\xi^\alpha_{x}(\eta_s)\xi^{\alpha+i}_{x+1}(\eta_s)$ with $ \int_\epsilon^{1-\epsilon}\nabla\phi_s(u) \langle \pi_s^{N,\alpha},\ola{\iota_\epsilon}(u)\rangle \langle \pi_s^{N,\alpha+i}, \ora{\iota_\epsilon}(u)\rangle du$ for $i=1,2$. 
 In order to do this, we sum and subtract  $\xi^\alpha_{x}(\eta_s)\langle \pi_s^{N,\alpha+i}, \ora{\iota_\epsilon}(\tfrac {x}{N})\rangle+\xi^{\alpha+i}_{x+1}(\eta_s)\langle \pi_s^{N,\alpha}, \ola{\iota_\epsilon}(\tfrac {x}{N})\rangle$.
Finally, we approximate $\nabla\phi$ with its discrete derivative $\nabla^+_N\phi(\tfrac{x}{N})$ and we observe that, since $\langle \pi_s^{N,\alpha},\ola{\iota_\epsilon}(\tfrac xN)\rangle  \langle \pi_s^{N,\alpha+1}, \ora{\iota_\epsilon}(\tfrac xN)\rangle = \ola\xi^{\alpha,\epsilon N}_x(\eta)\ora\xi^{\alpha+1,\epsilon N}_{x+1}(\eta)+\Or(\tfrac{1}{\epsilon N})$, if we reduce the sum to $\Lambda^\epsilon_N$, and we apply Lemma~\ref{lem:local_RL} to $\xi^\alpha_x(\eta)$ and $\xi^{\alpha+i}_{x+1}(\eta)$, we prove that~\eqref{eq:CLP9} goes to zero as $N\to\infty$ and $\epsilon\to0$. 

For $\Gamma={\rm Rob}$, by~\eqref{eq:weak_rob_4} other boundary terms will appear in~\eqref{eq:CLP1} and in addition to~\eqref{eq:CLP2}-~\eqref{eq:CLP5}. We prove that the probability of one of these terms goes to zero as $N\to\infty$ and $\epsilon\to0$. The rest is proved analogously. Let us consider, for example, $\int_0^t[\nabla\phi_s(0)\rho^\alpha_s(0)-\nabla\phi_s(1)\rho^\alpha_s(1)]ds$. We sum and subtract 
\begin{equation*}
 \int_0^t[\nabla\phi_s(0)\langle\pi^\alpha_s,\ora{\iota_\epsilon}(0)\rangle-\nabla\phi_s(1)\langle\pi^\alpha_s,\ola{\iota_\epsilon}(1)\rangle]ds
\end{equation*}
in $F^\alpha_{\rm Rob}(t,\phi,\rho,\mathfrak{g})$.
 As a consequence, we have to add in the absolute value of~\eqref{eq:CLP1}
 \begin{equation}\label{eq:CLP10}
  \kappa_1\int_0^t[\nabla\phi_s(0)(\rho^\alpha_s(0)-\langle\pi^\alpha_s,\ora{\iota_\epsilon}(0)\rangle)-\nabla\phi_s(1)(\rho^\alpha_s(1)-\langle\pi^\alpha_s,\ola{\iota_\epsilon}(1)\rangle)]ds
 \end{equation} and by~\eqref{eq:passage_rho} the contribution from these terms, vanish as $\epsilon\to0$. While~\eqref{eq:CLP10} will results into extra terms in~\eqref{eq:CLP11} and produce the terms
\begin{equation*}
 \begin{aligned}
  &\Pb_{\mu_N}\Bigg(\sup_{0\leq t\leq T}\Bigg|\kappa_1\int_0^t [\nabla_N^+\phi_s(\tfrac{1}{N})\xi^{\alpha}_1(\eta_s)-\nabla\phi_s(0)\langle\pi^{\alpha}_s,\ora{\iota_\epsilon}(0)\rangle] ds\Bigg|>\tilde\varepsilon\Bigg)\\
  &\Pb_{\mu_N}\Bigg(\sup_{0\leq t\leq T}\Bigg|\kappa_1\int_0^t [\nabla_N^+\phi_s(\tfrac{N-1}{N})\xi^{\alpha}(\eta_s)-\nabla\phi_s(1)\langle\pi^{\alpha}_s,\ora{\iota_\epsilon}(1)\rangle] ds\Bigg|>\tilde\varepsilon\Bigg).
 \end{aligned}
\end{equation*}
After approximating the continuous gradients with the discrete ones, these terms vanish as  $N\to\infty$ and $\epsilon\to0$ as a consequence of Lemma~\ref{lem:local_RL} and~\eqref{eq:passage_rho}.
\end{proof}

\section{Replacement lemmas and energy estimate}\label{sec:RL}

\subsection{Entropy bounds}\label{sec:entropy}

For a profile $\rho=(\rho^A,\rho^B,\rho^\emptyset):[0,1]\to[0,1]^3$ with $\rho^A+\rho^B+\rho^\emptyset=1$ let $\nu^N_{\rho(\cdot)}$ be the product measure on $\Omega_N$ whose marginals are given by
 \begin{equation}
 \label{eq:prod_measure}
 \nu^N_{\rho(\cdot)}(\eta(x)=\alpha)=\rho^\alpha(\tfrac xN), \quad \alpha\in\{A,B,\emptyset\}.
 \end{equation}
The first estimate we need is an upper bound on the relative entropy $H(\mu_N|\nu^N_{\rho(\cdot)})$ of the measure $\mu_N$ with respect to $\nu^N_{\rho(\cdot)}$.

\begin{lem}\label{lem:entropy} Let $\rho=(\rho^A,\rho^B,\rho^\emptyset):[0,1]\to[0,1]^3$ be a profile for which there exists some $r_0>0$ such that $\rho^\alpha(u)\geq r_0$ for each $u\in[0,1]$ and $\alpha\in\{A,B,\emptyset\}$. Then, there exists a constant $C_0$ such that $H(\mu_N|\nu^N_{\rho(\cdot)})\leq C_0N$ for every measure $\mu_N$ on $\Omega_N$.
\end{lem}
\begin{proof}
Observe that $\nu^N_{\rho(\cdot)}(\eta)=\prod_x\sum_{\alpha}\big(\rho^{\alpha}(\tfrac xN)\Big)^{\xi^\alpha_x(\eta)}\geq r_0^N$. Then, by the definition of relative entropy, we have that
\begin{equation}\label{eq:entropy}
\begin{aligned}
H(\mu_N|\nu^N_{\rho(\cdot)})&=\sum_{\eta\in\Omega_N}\mu_N(\eta)\log\left(\frac{\mu_N(\eta)}{\nu^N_{\rho(\cdot)}(\eta)}\right)\leq \sum_{\eta\in\Omega_N}\mu_N(\eta)\log\left(\frac{1}{\nu^N_{\rho(\cdot)}(\eta)}\right)\leq  N\log\frac{1}{r_0}\leq C_0 N.
\end{aligned}
\end{equation}
 \end{proof}
 \subsection{Dirichlet forms}\label{sec:DirichletForm}
 Recall \eqref{eq:prod_measure}.  
 We introduce the quadratic form, i.e. the non-negative functions associated to the bulk and the boundary generators given by:
 \begin{equation}\label{eq:DNB}
  \mcb D_N^B(\sqrt f,\nu^N_{\rho(\cdot)})=\sum_{x=1}^{N-2}\int_{\Omega_N}c_{x,x+1}(\eta)(\sqrt{f(\eta^{x,x+1})}-\sqrt{f(\eta)})^2 d\nu^N_{\rho(\cdot)}
 \end{equation}
 \begin{equation}\label{eq:DNL}
\begin{aligned}
  \mcb D_N^L(\sqrt f,\nu^N_{\rho(\cdot)})=\int_{\Omega_N}\!\!\Big\{ c^+_1(\eta)(\sqrt{f(\eta^{1,+})}-\sqrt{f(\eta)})^2+c^-_1(\eta)(\sqrt{f(\eta^{1,-})}-\sqrt{f(\eta)})^2\Big\}d\nu^N_{\rho(\cdot)}
  \end{aligned}
 \end{equation}
\begin{equation}\label{eq:DNR}
\begin{aligned}
  \mcb D_N^R(\sqrt f,\nu^N_{\rho(\cdot)})=\int_{\Omega_N}\Big\{ c^+_{N-1}(\eta)(\sqrt{f(\eta^{1,+})}-\sqrt{f(\eta)})^2+c^-_{N-1}(\eta)(\sqrt{f(\eta^{-})}-\sqrt{f(\eta)})^2\Big\}d\nu^N_{\rho(\cdot)}.
  \end{aligned}
 \end{equation}
 
 We will make use of the following lemmas.

\begin{lem}[Lemma 5.1 of~\cite{BGJO2019}]\label{lem:lem3.2}
 Let $T:\Omega_N\to\Omega_N$ be a transformation as the particle exchange between two sites $\eta^{x,x+1}$ or the particle upgrade/downgrade $\eta^{x,\pm}$. Let $c:\Omega_N\to\R$ be a positive function and $f$ a density with respect to $\mu$. Then
 \begin{equation*}
  \begin{aligned}
   \langle c(\eta)\big[\sqrt{f(T(\eta))}-&\sqrt{f(\eta)}\big],\sqrt{f(\eta)}\rangle_\mu\leq -\frac14\int c(\eta)\big[\sqrt{f(T(\eta))}-\sqrt{f(\eta)}\big]^2d\mu\\
   &+\frac{1}{16}\int \frac{1}{c(\eta)}\Big[c(\eta)-c(T(\eta))\frac{\mu(T(\eta))}{\mu(\eta)}\Big]^2\big(\sqrt{f(T(\eta))}+\sqrt{f(\eta)}\big)^2d\mu.
  \end{aligned}
 \end{equation*}
\end{lem}

\begin{lem}[Lemma 5.2 of~\cite{BGJO2019}]\label{lem:lem3.3}
 For $\alpha \in \{A,B,\emptyset\}$, let $\rho^\alpha:[0,1]\to[0,1]$ be a Lipschitz continuous functions such that  \begin{equation}\label{assu_1}
0<r_0\leq \rho^\alpha(u)\leq r_1<1\end{equation}
 for $u\in [0,1]$. Then, there exists a constant $C$ such that, for any $N\in\N$ and any density $f$ with respect to $\nu^N_{\rho(\cdot)}$,
 \begin{equation*}
  \sup_{1\leq x\leq N-2}\int_{\Omega_N} f(\eta^{x,x+1})d\nu^N_{\rho(\cdot)}\leq C,\quad \sup_{x\in\{1,N-1\}}\int_{\Omega_N}f(\eta^{x,\pm})d\nu^N_{\rho(\cdot)}\leq C.
 \end{equation*}
\end{lem}
As a consequence of Lemmas~\ref{lem:lem3.2} and \ref{lem:lem3.3} we obtain the following estimate on the Dirichlet form associated with  the bulk generator.
\begin{cor}\label{cor:DirichletLB}  For $\alpha \in \{A,B,\emptyset\}$, let $\rho^\alpha:[0,1]\to[0,1]$ be a Lipschitz continuous function satisfying \eqref{assu_1}. There exists a constant $C>0$ such that, for any density $f$ and any $N\in\N$,
 \begin{equation*}
  \langle \mcb L^B_N \sqrt f, \sqrt f\rangle_{\nu^N_{\rho(\cdot)}}\leq -\frac14\mcb D^B_N(\sqrt f,\nu^N_{\rho(\cdot)})+\frac CN.
 \end{equation*}
\end{cor}

\begin{proof}
By Lemma~\ref{lem:lem3.2} 
 \begin{equation}\label{eq:eq.cor1}
  \begin{aligned}
  &\langle \mcb L^B_N \sqrt f, \sqrt f\rangle_{\nu^N_{\rho(\cdot)}}\leq -\frac14\mcb D^B_N(\sqrt f,\nu^N_{\rho(\cdot)})\\
  &+\frac{1}{16}\sum_{x=1}^{N-2}\int\frac{1}{c_{x,x+1}(\eta)}\Big[c_{x,x+1}(\eta)-c_{x,x+1}(\eta^{x,x+1})\frac{\nu^N_{\rho(\cdot)}(\eta^{x,x+1})}{\nu^N_{\rho(\cdot)}(\eta)}\Big]^2\big(\sqrt{f(\eta^{x,x+1})}+\sqrt{f(\eta)}\big)^2 d\nu^N_{\rho(\cdot)}.
  \end{aligned}
 \end{equation}
Since the exchange rate $c_{x,x+1}$ depends on the species $\alpha\in\{A,B,\emptyset\}$, we can write the integral on the right hand side  of~\eqref{eq:eq.cor1} as
\begin{equation*}
\begin{aligned}
&\!\!\!\!\!\!\!\! \sum_{\alpha\in\{A,B,\emptyset\}}\!\!\!\!\sum_{\substack{\zeta:\,\zeta(x)=\alpha,\\ \zeta(x+1)=\alpha+1}}\!\!\!\!\!\!\frac{1}{c_{x,x+1}(\zeta)}\Big[c_{x,x+1}(\zeta)-c_{x,x+1}(\zeta^{x,x+1})\frac{\nu^N_{\rho(\cdot)}(\zeta^{x,x+1})}{\nu^N_{\rho(\cdot)}(\zeta)}\Big]^2\big(\sqrt{f(\zeta^{x,x+1})}+\sqrt{f(\zeta)}\big)^2 \nu^N_{\rho(\cdot)}(\zeta)\\
\!\!\!\!+&\!\!\!\!\!\!\!\! \sum_{\alpha\in\{A,B,\emptyset\}}\!\!\!\!\sum_{\substack{\zeta:\,\zeta(x)=\alpha,\\ \zeta(x+1)=\alpha+2}}\!\!\!\!\frac{1}{c_{x,x+1}(\zeta)}\Big[c_{x,x+1}(\zeta)-c_{x,x+1}(\zeta^{x,x+1})\frac{\nu^N_{\rho(\cdot)}(\zeta^{x,x+1})}{\nu^N_{\rho(\cdot)}(\zeta)}\Big]^2\big(\sqrt{f(\zeta^{x,x+1})}+\sqrt{f(\zeta)}\big)^2 \nu^N_{\rho(\cdot)}(\zeta),
\end{aligned}
 \end{equation*}
which equals
\begin{equation}\label{eq:eq.cor2}
\begin{aligned}
& \sum_{\alpha\in\{A,B,\emptyset\}}\sum_{\substack{\zeta:\,\zeta(x)=\alpha,\\ \zeta(x+1)=\alpha+1}}\frac{1}{1{-}\frac{\beta}{2N}}\Big[1{-}\frac{\beta}{2N}-\big(1{+}\frac{\beta}{2N}\big)\frac{\rho^{\alpha+1}(\tfrac xN)\rho^\alpha(\tfrac{x+1}{N})}{\rho^{\alpha}(\tfrac xN)\rho^{\alpha+1}(\tfrac{x+1}{N})}\Big]^2\big(\sqrt{f(\zeta^{x,x+1})}+\sqrt{f(\zeta)}\big)^2 \nu^N_{\rho(\cdot)}(\zeta)\\
+& \sum_{\alpha\in\{A,B,\emptyset\}}\sum_{\substack{\zeta:\,\zeta(x)=\alpha,\\ \zeta(x+1)=\alpha+2}}\frac{1}{1{+}\frac{\beta}{2N}}\Big[1{+}\frac{\beta}{2N}-\big(1{-}\frac{\beta}{2N}\big)\frac{\rho^{\alpha+2}(\tfrac xN)\rho^\alpha(\tfrac{x+1}{N})}{\rho^{\alpha}(\tfrac xN)\rho^{\alpha+2}(\tfrac{x+1}{N})}\Big]^2\big(\sqrt{f(\zeta^{x,x+1})}+\sqrt{f(\zeta)}\big)^2 \nu^N_{\rho(\cdot)}(\zeta),
\end{aligned}
 \end{equation}
 where in the last passage we used the fact that $\nu^N_{\rho(\cdot)}$ is a product measure. 
For fixed $\alpha$, we rearrange the terms in the square brackets: in the first term we get
\begin{multline}\label{eq:corol:bulk}
\frac{1}{\big(\rho^{\alpha}(\tfrac xN)\rho^{\alpha+1}(\tfrac{x+1}{N})\big)^2}\Big[\rho^{\alpha}(\tfrac xN)\rho^{\alpha+1}(\tfrac{x+1}{N})-\rho^{\alpha+1}(\tfrac xN)\rho^{\alpha}(\tfrac{x+1}{N})\\{-}\frac{\beta}{2N}\big(\rho^{\alpha+1}(\tfrac xN)\rho^{\alpha}(\tfrac{x+1}{N})+\rho^{\alpha}(\tfrac xN)\rho^{\alpha+1}(\tfrac{x+1}{N})\big)\Big]^2
\end{multline}
Note that
\begin{equation*}
\begin{aligned}
&\rho^{\alpha}(\tfrac xN)\rho^{\alpha+1}(\tfrac{x+1}{N})-\rho^{\alpha+1}(\tfrac xN)\rho^{\alpha}(\tfrac{x+1}{N})\\
=&\rho^{\alpha}(\tfrac xN)(\rho^{\alpha+1}(\tfrac{x+1}{N})-\rho^{\alpha+1}(\tfrac xN))-\rho^{\alpha+1}(\tfrac{x}{N})(\rho^{\alpha}(\tfrac{x+1}{N})-\rho^{\alpha}(\tfrac{x}{N})).
\end{aligned}
\end{equation*}
Since for each $\alpha$, $\rho^\alpha(\cdot)$ is Lipschitz and satisfies \eqref{assu_1}, \eqref{eq:corol:bulk} is bounded from above by $\frac{C}{N^2}$. For the second term of~\eqref{eq:eq.cor2} we obtain the same bound. And so we can bound  
\begin{equation*}
 \begin{aligned}
 \sum_{x=1}^{N-2}\eqref{eq:eq.cor2}\leq&\sum_{x=1}^{N-2} \sum_{\alpha\in\{A,B,\emptyset\}}\sum_{\substack{\zeta:\,\zeta(x)=\alpha,\\ \zeta(x+1)=\alpha+1}}\frac{\widetilde C}{N^2}\big(\sqrt{f(\zeta^{x,x+1})}+\sqrt{f(\zeta)}\big)^2\nu^N_{\rho(\cdot)}(\zeta)\\
  +&  \sum_{x=1}^{N-2}\sum_{\alpha\in\{A,B,\emptyset\}}\sum_{\substack{\zeta:\,\zeta(x)=\alpha,\\ \zeta(x+1)=\alpha+2}}\frac{\widetilde C}{N^2}(\sqrt{f(\zeta^{x,x+1})}+\sqrt{f(\zeta)}\big)^2\nu^N_{\rho(\cdot)}(\zeta).
 \end{aligned}
\end{equation*}
Using the fact that $(a+b)^2\leq2(a^2+b^2)$, since $f$ is a density with respect to $\nu^N_{\rho(\cdot)}$ we get, by Lemma \ref{lem:lem3.3}, that
  \begin{equation*}
  \langle \mcb L^B_N \sqrt f, \sqrt f\rangle_{\nu^N_{\rho(\cdot)}}\leq -\frac14\mcb D^B_N(\sqrt f,\nu^N_{\rho(\cdot)})+\frac CN.
 \end{equation*}
\end{proof}
As a corollary of Lemma~\ref{lem:lem3.3} we obtain an estimate on the Dirichlet forms associated with the boundary generators.

\begin{cor}\label{cor:DirichletLR} For each $\alpha\in\{A,B,\emptyset\}$, let $\rho^\alpha:[0,1]\to[0,1]$ be a Lipschitz continuous function satisfying \eqref{assu_1}  and also \begin{equation}\label{assu_2} 
\rho^\alpha(u)=r_\alpha, \forall u\in[0,\varepsilon]\quad \textrm{and}\quad \rho^\alpha(u)=\tilde r_\alpha, \forall u\in[1-\varepsilon,1],
\end{equation}
for some $\varepsilon>0$ small. Then, there exists a constant $C>0$ such that, for any density $f$ and any $N$ sufficiently big, it holds
\begin{equation*}
  \langle \mcb L^L_N \sqrt f, \sqrt f\rangle_{\nu^N_{\rho(\cdot)}}\leq -\frac12\mcb D^L_N(\sqrt f,\nu^N_{\rho(\cdot)})+ \frac{C}{N^\theta},
 \end{equation*}
 \begin{equation*}
  \langle \mcb L^R_N \sqrt f, \sqrt f\rangle_{\nu^N_{\rho(\cdot)}}\leq -\frac12\mcb D^R_N(\sqrt f,\nu^N_{\rho(\cdot)})+\frac{C}{N^\theta}
 \end{equation*}
\end{cor}
\begin{proof}
 We prove only the inequality for the left boundary generator. The right boundary generator it follows identically. From the identity $(a-b)b=-\frac12(a-b)^2+\frac12(a^2-b^2)$ it follows that
  \begin{align}
  &\langle \mcb L^L_N \sqrt f, \sqrt f\rangle_{\nu^N_{\rho(\cdot)}}\nonumber\\=&\int c_{1}^+(\eta)\big[\sqrt{f(\eta^{1,+})}\!-\!\sqrt{f(\eta)}\big]\sqrt{f(\eta)}d\nu^N_{\rho(\cdot)}\!+\!\int c_{1}^-(\eta)\big[\sqrt{f(\eta^{1,-})}\!-\!\sqrt{f(\eta)}\big]\sqrt{f(\eta)}d\nu^N_{\rho(\cdot)}\nonumber\\
  =&-\frac12\int c_{1}^+(\eta)\big[\sqrt{f(\eta^{1,+})}-\sqrt{f(\eta)}\big]^2d\nu^N_{\rho(\cdot)}-\frac12\int c_{1}^-(\eta)\big[\sqrt{f(\eta^{1,-})}-\sqrt{f(\eta)}\big]^2d\nu^N_{\rho(\cdot)} \label{eq:eq3.32}\\
  &+\frac12\int c_1^+(\eta)\big[f(\eta^{1,+})-f(\eta)\big]d\nu^N_{\rho(\cdot)}\label{eq:eq3.33} +\frac12\int c_1^-(\eta)\big[f(\eta^{1,-})-f(\eta)\big]d\nu^N_{\rho(\cdot)} 
  \end{align}
  The term in~\eqref{eq:eq3.32} is equal to $-\frac 12\mcb D^L_N$ in~\eqref{eq:DNL}, so we have to estimate~\eqref{eq:eq3.33}.
 We can perform a change of variable in the second integral and  we can rewrite \eqref{eq:eq3.33} as
 \begin{align*}
 &\frac12\int [f(\eta^{1,+})-f(\eta)\big]\Big[c^+_1(\eta)-c^-_1(\eta^{1,+})\frac{\nu^N_{\rho(\cdot)}(\eta^{1,+})}{\nu^N_{\rho(\cdot)}(\eta)}\Big]d\nu^N_{\rho(\cdot)}\\
 &=\frac12\sum_{\alpha}\sum_{\eta: \eta(1)=\alpha}[f(\eta^{1,+})-f(\eta)\big]\Big[\left(\frac{1}{N^\delta}+\frac{\tilde \beta}{2N^\theta}\right)r_{\alpha+1}-\left(\frac{1}{N^\delta}-\frac{\tilde \beta}{2N^\theta}\right)r_{\alpha}\frac{\rho^{\alpha+1}(\tfrac{1}{N})}{\rho^\alpha(\tfrac{1}{N})}\Big]\nu^N_{\rho(\cdot)}(\eta).
 \end{align*}
 We observe that the terms proportional to $1/N^\delta$ are equal to zero for $N$ sufficiently big,  since we  choose a profile satisfying \eqref{assu_2}. The remaining terms are equal to
 \begin{equation*}
 \frac{\tilde\beta}{2N^\theta}\sum_{\alpha\in\{A,B,\emptyset\}}\,\sum_{\eta:\eta(1)=\alpha}[f(\eta^{1,+})-f(\eta)]r_{\alpha+1}\nu^N_{\rho(\cdot)}(\eta)\leq \frac{\tilde\beta}{2N^\theta}\int[f(\eta^{1,+})-f(\eta)]d\nu^N_{\rho(\cdot)}\leq \frac{C}{N^\theta},
 \end{equation*}
 by Lemma \ref{lem:lem3.3}. 
\end{proof}

Now that we have the ingredients we need, in the next subsections we prove all the replacement lemmas that were used in the previous section on the characterization of limit points.

\subsection{Replacement lemma at the bulk}\label{sec:RLbulk}

\begin{lem}[Local replacement lemma]\label{lem:local_RL}

\quad

For any $t\in[0,T]$ and $\theta\geq 1$, any $x \in \{1,\ldots,N-1-\epsilon N\}$ we have
\begin{equation}
 \lim_{\epsilon\to0}\lim_{N\to\infty}\E_{\mu_N}\left[\left|\int_0^t (\xi^\alpha_x(\eta_s)-\overrightarrow{\xi}^{\alpha,\epsilon N}_x(\eta_s))ds\right|\right]=0.
\end{equation}
The same result above holds by replacing $\overrightarrow{\xi}^{\alpha,\epsilon N}_x(\eta_s)$ with $\overleftarrow{\xi}^{\alpha,\epsilon N}_x(\eta_s)$  for $x \in \{\epsilon N,\ldots,N-1\}$.
\end{lem}
\begin{proof}
Consider the product measure $\nu^N_{\rho(\cdot)}$ as given in \eqref{eq:prod_measure} where the profile $\rho(\cdot)$ is Lipschitz continuous and  satisfying \eqref{assu_1} and \eqref{assu_2}.
 By the entropy and Jensen's inequality, we can bound the expected value in last display by 
 \begin{equation}\label{eq:RLbulk1}
  \frac{H(\mu_N|\nu_{\rho(\cdot)}^N)}{BN}+\frac{1}{NB}\log\E_{\nu^N_{\rho(\cdot)}}\left[e^{NB|\int_0^t (\xi^\alpha_x(\eta_s)-\overrightarrow{\xi}^{\alpha,\epsilon N}_x(\eta_s))|}ds\right],
 \end{equation}
 where $B>0$.
We can remove the absolute value in the exponential, since $e^{|x|}\leq e^x+e^{-x}$ and
\begin{equation*}
\limsup_{N\to\infty}\frac{1}{N}\log(a_N+b_N)\leq \max\Big\{\limsup_{N\to\infty}\frac{1}{N}\log a_N,\limsup_{N\to\infty}\frac{1}{N}\log b_N\Big\}.
\end{equation*}
From Lemma~\ref{lem:entropy} and from Feynman--Kac's formula we can bound~\eqref{eq:RLbulk1} from above by 
\begin{equation}\label{eq:RLbulkF}
\frac{C_0}{B}+t\sup_f\left\{\langle \xi^\alpha_x(\eta)-\overrightarrow{\xi}^{\alpha,\epsilon N}_x(\eta),f\rangle_{\nu^N_{\rho(\cdot)}} + \frac{N}{B}\langle \mcb L_N\sqrt{f},\sqrt{f}\rangle_{\nu^N_{\rho(\cdot)}}\right\} 
\end{equation}
where the supremum is taken over all densities $f$ with respect to $\nu^N_{\rho(\cdot)}$. From Corollaries~\ref{cor:DirichletLB} and~\ref{cor:DirichletLR}, this supremum is bounded from above by
\begin{multline}\label{eq:RLbulk4}
\sup_f\Big\{\langle \xi^\alpha_x(\eta)-\overrightarrow{\xi}^{\alpha,\epsilon N}_x(\eta),f\rangle_{\nu^N_{\rho(\cdot)}} - \frac{N}{4B} {\mcb D}_N^B(\sqrt{f},\nu^N_{\rho(\cdot)})\\
- \frac{N}{2B} {\mcb D}_N^L(\sqrt{f},\nu^N_{\rho(\cdot)})- \frac{N}{2B} {\mcb D}_N^R(\sqrt{f},\nu^N_{\rho(\cdot)})+\frac{ C N}{B N^\theta}+\frac{\tilde C}{B}\Big\} .
\end{multline}
By writing  the first term between brackets as a telescopic sum we get that
\begin{equation*}
 \langle \xi^\alpha_x(\eta)-\overrightarrow{\xi}^{\alpha,\epsilon N}_x(\eta),f\rangle_{\nu^N_{\rho(\cdot)}}=\frac{1}{\epsilon N}\int_{\Omega_N}\sum_{y=x+2}^{x+\epsilon N}\sum_{z=x+1}^{y-1}(\xi^\alpha_z(\eta)-\xi^\alpha_{z+1}(\eta))f(\eta)d\nu^N_{\rho(\cdot)}.
\end{equation*}
We split the above integral into two identical terms and we sum and subtract the term
\begin{equation*}
 \frac{1}{2\epsilon N}\int_{\Omega_N}\sum_{y=x+2}^{x+\epsilon N}\sum_{z=x+1}^{y-1}(\xi^\alpha_z(\eta)-\xi^\alpha_{z+1}(\eta))f(\eta^{z,z+1})d\nu^N_{\rho(\cdot)}.
\end{equation*}
Rearranging the terms we have
 \begin{align}
 \langle \xi^\alpha_x(\eta)-\overrightarrow{\xi}^{\alpha,\epsilon N}_x(\eta),f\rangle_{\nu^N_{\rho(\cdot)}}=
 &\frac{1}{2\epsilon N}\int_{\Omega_N}\sum_{y=x+2}^{x+\epsilon N}\sum_{z=x+1}^{y-1}(\xi^\alpha_z(\eta)-\xi^\alpha_{z+1}(\eta))[f(\eta)-f(\eta^{z,z+1})]d\nu^N_{\rho(\cdot)} \label{eq:RLbulk2}\\
 +&
 \frac{1}{2\epsilon N}\int_{\Omega_N}\sum_{y=x+2}^{x+\epsilon N}\sum_{z=x+1}^{y-1}(\xi^\alpha_z(\eta)-\xi^\alpha_{z+1}(\eta))[f(\eta)+f(\eta^{z,z+1})]d\nu^N_{\rho(\cdot)}. \label{eq:RLbulk3}
 \end{align}
We start by estimating the first term. We multiply and divide it by $\sqrt{c_{z,z+1}(\eta)}$ and apply Young's inequality with a constant $A>0$ to bound it from above by
\begin{equation*}
 \begin{aligned}
 &\frac{1}{2\epsilon N}\frac{1}{2A}\int_{\Omega_N}\sum_{y=x+2}^{x+\epsilon N}\sum_{z=x+1}^{y-1}(\xi^\alpha_z(\eta)-\xi^\alpha_{z+1}(\eta))^2\frac{1}{c_{z,z+1}(\eta)}[\sqrt{f(\eta)}+\sqrt{f(\eta^{z,z+1})}]^2d\nu^N_{\rho(\cdot)}\\
  + &\frac{1}{2\epsilon N}\frac{A}{2}\int_{\Omega_N}\sum_{y=x+2}^{x+\epsilon N}\sum_{z=x+1}^{y-1}{c_{z,z+1}(\eta)}[\sqrt{f(\eta)}-\sqrt{f(\eta^{z,z+1})}]^2d\nu^N_{\rho(\cdot)}.
 \end{aligned}
\end{equation*}
Choosing $A=N/B$ the first term goes to zero as $N\to\infty$ and $\epsilon\to0$, while the second term cancels with the bulk Dirichlet form $- \frac{N}{4B} \mcb D_N^B(\sqrt{f},\nu^N_{\rho(\cdot)})$  in~\eqref{eq:RLbulk4}. Now we estimate~\eqref{eq:RLbulk3}. We split it as
 \begin{align}
 &\frac{1}{2\epsilon N}\int_{\Omega_N}\sum_{y=x+2}^{x+\epsilon N}\sum_{z=x+1}^{y-1}(\xi^\alpha_z(\eta)-\xi^\alpha_{z+1}(\eta))f(\eta)d\nu^N_{\rho(\cdot)} \label{eq:RLbulk5}\\
 +&
 \frac{1}{2\epsilon N}\int_{\Omega_N}\sum_{y=x+2}^{x+\epsilon N}\sum_{z=x+1}^{y-1}(\xi^\alpha_z(\eta)-\xi^\alpha_{z+1}(\eta))f(\eta^{z,z+1})d\nu^N_{\rho(\cdot)}. \label{eq:RLbulk6}
 \end{align}
In the second term we perform the change of variables $\eta\mapsto\eta^{z,z+1}$ and combine again the integrals to obtain
\begin{equation*}
  \eqref{eq:RLbulk3}=\frac{1}{2\epsilon N}\int_{\Omega_N}\sum_{y=x+2}^{x+\epsilon N}\sum_{z=x+1}^{y-1}(\xi^\alpha_z(\eta)-\xi^\alpha_{z+1}(\eta))\Big[1-\frac{\nu^N_{\rho(\cdot)}(\eta^{z,z+1})}{\nu^N_{\rho(\cdot)}(\eta)}\Big]f(\eta)d\nu^N_{\rho(\cdot)}. 
\end{equation*}
Now we condition  on having either one of the three species at $z$ or at $z+1$. We exclude the configurations $\beta\beta$, for any $\beta$, $(\alpha+1)(\alpha+2)$ and $(\alpha+2)(\alpha+1)$ that give contribution zero. The remaining configurations to check are $(\alpha)(\alpha+i)$ and $(\alpha+i)(\alpha)$, $i=1,2$. So the previous term becomes equal to
\begin{equation*}
\begin{aligned}
 &\frac{1}{2\epsilon N}\sum_{i=1}^2\Bigg(\int_{\Omega_N}\sum_{y=x+2}^{x+\epsilon N}\sum_{z=x+1}^{y-1}\mathbbm{1}_{[\eta(z)=\alpha,\eta(z+1)=\alpha+i]}(\xi^\alpha_z(\eta)-\xi^\alpha_{z+1}(\eta))(\nu^N_{\rho(\cdot)}(\eta)-\nu^N_{\rho(\cdot)}(\eta^{z,z+1}))f(\eta)d\nu^N_{\rho(\cdot)}\\
 +&\int_{\Omega_N}\sum_{y=x+2}^{x+\epsilon N}\sum_{z=x+1}^{y-1}\mathbbm{1}_{[\eta(z)=\alpha+i,\eta(z+1)=\alpha]}(\xi^\alpha_z(\eta)-\xi^\alpha_{z+1}(\eta))(\nu^N_{\rho(\cdot)}(\eta)-\nu^N_{\rho(\cdot)}(\eta^{z,z+1}))f(\eta)d\nu^N_{\rho(\cdot)}\Bigg).
 \end{aligned}
\end{equation*}
Let $\eta=(\bar\eta,\eta_z,\eta_{z+1})$ where $\bar\eta$ is the particle configuration $(\eta_1,\eta_2,\dots,\hat\eta_z,\hat\eta_{z+1},\dots,\eta_{N-1})$. Now the previous term can be written as
\begin{equation}\label{eq:RLbulk7}
\begin{aligned}
 &\frac{1}{2\epsilon N}\sum_{i=1}^2\Bigg(\int_{\Omega_N}\sum_{y=x+2}^{x+\epsilon N}\sum_{z=x+1}^{y-1}\mathbbm{1}_{[\eta(z)=\alpha,\eta(z+1)=\alpha+i]}(\nu^N_{\rho(\cdot)}(\bar\eta,\alpha,\alpha+i)-\nu^N_{\rho(\cdot)}(\bar\eta,\alpha+i,\alpha))f(\bar\eta,\alpha,\alpha+i)d\nu^N_{\rho(\cdot)}\\
 -&\int_{\Omega_N}\sum_{y=x+2}^{x+\epsilon N}\sum_{z=x+1}^{y-1}\mathbbm{1}_{[\eta(z)=\alpha+i,\eta(z+1)=\alpha]}(\nu^N_{\rho(\cdot)}(\bar\eta,\alpha+i,\alpha)-\nu^N_{\rho(\cdot)}(\bar\eta,\alpha,\alpha+i))f(\bar\eta,\alpha+i,\alpha)d\nu^N_{\rho(\cdot)}\Bigg)\\
 =&\frac{1}{2\epsilon N}\sum_{i=1}^2\int_{\Omega_N}\sum_{y=x+2}^{x+\epsilon N}\sum_{z=x+1}^{y-1}(\nu^N_{\rho(\cdot)}(\bar\eta,\alpha,\alpha+i)-\nu^N_{\rho(\cdot)}(\bar\eta,\alpha+i,\alpha))[\mathbbm{1}_{[\eta(z)=\alpha,\eta(z+1)=\alpha+i]}f(\bar\eta,\alpha,\alpha+i)\\
 &\qquad\qquad\qquad\qquad\qquad\qquad\qquad\qquad\qquad\qquad+\mathbbm{1}_{[\eta(z)=\alpha+i,\eta(z+1)=\alpha]}f(\bar\eta,\alpha+i,\alpha)]d\nu^N_{\rho(\cdot)}.
 \end{aligned}
\end{equation}
We have to estimate the contribution given by the difference of the measures:
\begin{equation}\label{eq:diff_measures}
 \begin{aligned}
  &\nu^N_{\rho(\cdot)}(\bar\eta,\alpha,\alpha+i)-\nu^N_{\rho(\cdot)}(\bar\eta,\alpha+i,\alpha)\\
  =&\prod_{x\neq z,z+1}\rho^{\eta_x}(\tfrac xN)[\rho^\alpha(\tfrac{z}{N})\rho^{\alpha+i}(\tfrac{z+1}{N})-\rho^{\alpha+i}(\tfrac{z}{N})\rho^{\alpha}(\tfrac{z+1}{N})]\\
  =&\prod_{x\neq z,z+1}\rho^{\eta_x}(\tfrac xN)[\rho^\alpha(\tfrac{z}{N})(\rho^{\alpha+i}(\tfrac{z+1}{N})-\rho^{\alpha+i}(\tfrac{z}{N}))-\rho^{\alpha+i}(\tfrac{z}{N})(\rho^{\alpha}(\tfrac{z+1}{N})-\rho^{\alpha}(\tfrac{z}{N}))]\\
  =&\nu^N_{\rho(\cdot)}(\bar\eta)[\rho^\alpha(\tfrac{z}{N})(\rho^{\alpha+i}(\tfrac{z+1}{N})-\rho^{\alpha+i}(\tfrac{z}{N}))-\rho^{\alpha+i}(\tfrac{z}{N})(\rho^{\alpha}(\tfrac{z+1}{N})-\rho^{\alpha}(\tfrac{z}{N}))]
  \simeq\frac{1}{N}\nu^N_{\rho(\cdot)}(\bar\eta), 
 \end{aligned}
\end{equation}
if we impose a Lipschitz condition on the density profiles of the three species.
In the end,~\eqref{eq:RLbulk7} is of order $\frac{1}{\epsilon N}(\epsilon N)^2\frac{1}{N}$ which goes to zero as $\epsilon\to0$. 

Now,  \eqref{eq:RLbulkF}  can be bounded from above by
\begin{equation}\label{eq:RLbulkFi}
\frac{C_0}{B}+\frac{ C N}{B N^\theta}+\frac{\tilde C}{B}+\epsilon. 
\end{equation}
Since $\theta\geq 1$, taking the limit in $N\to+\infty$, then $\epsilon \to 0$ and finally $B\to+\infty$, we are done. 

\end{proof}

By adapting the proof of last result, we can also derive the local replacement lemma which is useful to treat the contribution in Dynkin's formula coming from the antisymmetric part of the dynamics. 

\begin{lem}[Global replacement lemma]\label{eq:global_RL}

\quad
Let $G:[0,T]\times[0,1]\to\mathbb R$ be a bounded function and let $\psi:\Omega_N\to\mathbb R$  be a function such that for $x\in\{\epsilon N,\cdots, N-1-\epsilon N\}$, $\tau_x\psi(\eta)$ is invariant for the change of variables $\eta\mapsto\eta^{z,z+1}$ for $z\in\{x+1,\cdots, x+\epsilon N\}$. Then,
for any $t\in[0,T]$ and $\theta\geq 1$, we have
\begin{equation*}
 \lim_{\epsilon\to0}\lim_{N\to\infty}\E_{\mu_N}\left|\int_0^t \frac{1}{N}\sum_{x=\epsilon N}^{N-1-\epsilon N}G(s,\tfrac xN)\tau_x\psi(\eta_s)(\xi^\alpha_x(\eta_s)-\overrightarrow{\xi}^{\alpha,\epsilon N}_x(\eta_s))ds\right|=0.
\end{equation*}
\end{lem}
The same result holds with the right average replaced with the left average as long as for $x\in\{\epsilon N,\cdots, N-1-\epsilon N\}$, $\tau_x\psi(\eta)$ is invariant for the change of variables $\eta\mapsto\eta^{z,z-1}$ for $z\in\{x-\epsilon N,\cdots, x-1\}$.

\subsection{Replacement lemma at the boundary}\label{sec:RLboundary}

We start this subsection by computing the adjoint $\mcb L^{L,\star}_N$ of the left boundary generator with respect to the measure $\nu^N_{\rho(\cdot)}$ given in \eqref{eq:prod_measure}. The computations for the right boundary generator are completely analogous and we omit them. Note that  by the definition of $\mcb L^L_N$ in~\eqref{eq:Lgen}, for any $f,g\in\mathbb L^2(\nu^N_{\rho(\cdot)})$ we have
  \begin{align}\label{eq:termsAB}
  & \int \mcb L^L_N f(\eta)g(\eta)d\nur\\
 =& \int  \big(\frac{1}{N^\delta}+\frac{\tilde\beta}{2N^\theta}\big)(\xi^A_1(\eta)r_{B}+\xi^B_1(\eta)r_{\emptyset}+\xi^\emptyset_1(\eta)r_{A})[f(\eta^{1,+})-f(\eta)]g(\eta)d\nur \label{eq:termA}\\ 
 & +\int \big(\frac{1}{N^\delta}-\frac{\tilde\beta}{2N^\theta}\big)(\xi^A_1(\eta)r_{\emptyset}+\xi^B_1(\eta)r_{A}+\xi^\emptyset_1(\eta)r_{B})[f(\eta^{1,-})-f(\eta)]g(\eta)d\nur \label{eq:termB}
 \end{align}
We look at~\eqref{eq:termA} first.   We perform the change of variables $\zeta=\eta^{1,+}$ (equivalently $\eta=\zeta^{1,-}$) in the term with $f(\eta^{1,+})g(\eta)$ to rewrite that term as 
  \begin{align*}
& \int  \big(\frac{1}{N^\delta}+\frac{\tilde\beta}{2N^\theta}\big)(\xi^A_1(\zeta^{1,-})r_{B}+\xi^B_1(\zeta^{1,-})r_{\emptyset}+\xi^\emptyset_1(\zeta^{1,-})r_{A})  f(\zeta)g(\zeta^{1,-})\frac{\nur(\zeta^{1,-})}{\nur(\zeta)}d\nur\\
 =& \int \big(\frac{1}{N^\delta}+\frac{\tilde\beta}{2N^\theta}\big)(\xi^B_1(\zeta)r_{B}+\xi^\emptyset_1(\zeta)r_{\emptyset}+\xi^A_1(\zeta)r_{A}) f(\zeta)g(\zeta^{1,-})\frac{\nur(\zeta^{1,-})}{\nur(\zeta)}d\nur.
 \end{align*}
Now we split the integral by conditioning on having a particle of species $A,B$ or a hole at position $1/N$, so that last display is equal to 
  \begin{align*}
& \sum_\alpha\int \mathbbm{1}_{\{\zeta(1)=\alpha\}} \big(\frac{1}{N^\delta}+\frac{\tilde\beta}{2N^\theta}\big)r_{\alpha}f(\zeta)g(\zeta^{1,-})\frac{\rho^{\alpha+2}(\tfrac 1N)}{\rho^\alpha(\tfrac 1N)}d\nur.
 \end{align*}
In a similar way we rewrite  the term $f(\eta^{1,-})g(\eta)$ in   \eqref{eq:termB} as
  \begin{align*}
& \sum_\alpha\int \mathbbm{1}_{\{\zeta(1)=\alpha\}} \big(\frac{1}{N^\delta}-\frac{\tilde\beta}{2N^\theta}\big)r_{\alpha}f(\zeta)g(\zeta^{1,+})\frac{\rho^{\alpha+1}(\tfrac 1N)}{\rho^\alpha(\tfrac 1N)}d\nur
 \end{align*}
 Putting together the last two terms we can write them as 
   \begin{align*}
  & \int  \big(\frac{1}{N^\delta}+\frac{\tilde\beta}{2N^\theta}\big)(\xi^A_1(\eta)\frac{\rho^\emptyset(\tfrac 1N)}{\rho^A(\tfrac 1N)}r_{A}+\xi^B_1(\eta)\frac{\rho^A(\tfrac 1N)}{\rho^B(\tfrac 1N)}r_{B}+\xi^\emptyset_1(\eta)\frac{\rho^B(\tfrac 1N)}{\rho^\emptyset(\tfrac 1N)}r_{\emptyset}) f(\eta)g(\eta^{1,-})d\nur\\ 
+& \int  \big(\frac{1}{N^\delta}-\frac{\tilde\beta}{2N^\theta}\big)(\xi^A_1(\eta)\frac{\rho^B(\tfrac 1N)}{\rho^A(\tfrac 1N)}r_{A}+\xi^B_1(\eta)\frac{\rho^\emptyset(\tfrac 1N)}{\rho^B(\tfrac 1N)}r_{B}+\xi^\emptyset_1(\eta)\frac{\rho^A(\tfrac 1N)}{\rho^\emptyset(\tfrac 1N)}r_{\emptyset}) f(\eta)g(\eta^{1,+})d\nur
\end{align*}
Assuming \eqref{assu_1} and \eqref{assu_2} we can rewrite the last two terms, for $N$ sufficiently big as 
  \begin{align*}
  & \int  \big(\frac{1}{N^\delta}+\frac{\tilde\beta}{2N^\theta}\big)(\xi^A_1(\eta)r_{\emptyset}+\xi^B_1(\eta)r_{A}+\xi^\emptyset_1(\eta)r_{B}) f(\eta)g(\eta^{1,-})d\nur\\ 
+& \int  \big(\frac{1}{N^\delta}-\frac{\tilde\beta}{2N^\theta}\big)(\xi^A_1(\eta)r_{B}+\xi^B_1(\eta)r_{\emptyset}+\xi^\emptyset_1(\eta)r_{A}) f(\eta)g(\eta^{1,+})d\nur.
\end{align*}
Putting now all the terms together we get
 \begin{align*}
  \int \mcb L^L_N f(\eta)&g(\eta)d\nur\\
&=   \int  \big(\frac{1}{N^\delta}+\frac{\tilde\beta}{2N^\theta}\big)(\xi^A_1(\eta)r_{\emptyset}+\xi^B_1(\eta)r_{A}+\xi^\emptyset_1(\eta)r_{B}) f(\eta)g(\eta^{1,-})d\nur\\ 
&+ \int  \big(\frac{1}{N^\delta}-\frac{\tilde\beta}{2N^\theta}\big)(\xi^A_1(\eta)r_{B}+\xi^B_1(\eta)r_{\emptyset}+\xi^\emptyset_1(\eta)r_{A}) f(\eta)g(\eta^{1,+})d\nur\\
&- \int  \big(\frac{1}{N^\delta}+\frac{\tilde\beta}{2N^\theta}\big)(\xi^A_1(\eta)r_{B}+\xi^B_1(\eta)r_{\emptyset}+\xi^\emptyset_1(\eta)r_{A}) f(\eta)g(\eta)d\nur\\
&- \int  \big(\frac{1}{N^\delta}-\frac{\tilde\beta}{2N^\theta}\big)(\xi^A_1(\eta)r_{\emptyset}+\xi^B_1(\eta)r_{A}+\xi^\emptyset_1(\eta)r_{B}) f(\eta)g(\eta)d\nur
\end{align*}
for $N$ sufficiently big.  From this, we obtain  the following expression for the adjoint of $\mcb L^{L}_N$:
\begin{equation}\label{eq:Lgen_adj}
 \begin{aligned}
  \mcb L^{L,\star}_N g(\eta)=& \big(\frac{1}{N^\delta}-\frac{\tilde\beta}{2N^\theta}\big)(\xi^A_1(\eta)r_{B}+\xi^B_1(\eta)r_{\emptyset}+\xi^\emptyset_1(\eta)r_{A})[g(\eta^{1,+})-g(\eta)]\\
  +&\big(\frac{1}{N^\delta}+\frac{\tilde\beta}{2N^\theta}\big)(\xi^A_1(\eta)r_{\emptyset}+\xi^B_1(\eta)r_{A}+\xi^\emptyset_1(\eta)r_{B})[g(\eta^{1,-})-g(\eta)]\\
  +&\frac{\tilde\beta}{N^\theta}[\xi^A_1(\eta)(r_\emptyset-r_{B})+\xi^B_1(\eta)(r_A-r_{\emptyset})+\xi^\emptyset_1(\eta)(r_B-r_{A})]g(\eta).
 \end{aligned}
\end{equation}

Now we are ready to prove the replacement needed for the Dirichlet regime, i.e. for  item a) of Theorem 
\ref{th:hyd_ssep}.
\begin{lem} \label{lem:RL_boundary}
 For $t\in[0,T]$, $\theta\geq 1>\delta$ and for $\alpha\in\{A,B,\emptyset\}$, we have
 \begin{equation}\label{eq:RLleft}
 \lim_{N\to\infty}\E_{\mu_N}\left[\left|\int_0^t(r_{\alpha}-\xi^\alpha_1(\eta_s))ds\right|\right]=0.
 \end{equation}
The same holds replacing $r_\alpha$ with $\tilde r_\alpha$ and  $\xi^\alpha_1(\eta_s)$ with $ \xi^\alpha_{N-1}(\eta_s)$.
\end{lem}
\begin{proof}
We present the proof for the case $\alpha=A$ and for the left boundary, but for the other cases it is analogous. 
First we make the following key observation. By~\eqref{eq:Lgen_adj}, we have
\begin{equation*}
 \begin{aligned}
\mcb L^{L,\star}_N \xi^A_1(\eta)=\frac{1}{N^\delta}(r_A-\xi^A_1(\eta))+\frac{\tilde \beta}{2N^\theta}[r_A(\xi_1^B(\eta)-\xi_1^\emptyset(\eta))+(r_\emptyset-r_B)\xi_1^A(\eta)].
 \end{aligned}
\end{equation*}

As a consequence, and since we assumed $\theta\geq 1>\delta$, the second term on the right hand-side of last display{, multiplied by $N^\delta$,} goes to $0$ as $N\to\infty$ and we are left with the estimate of  
\begin{equation*}
 \E_{\mu_N}\left[\left|\int_0^t N^\delta\mcb L^{L,\star}_N \xi^A_1(\eta_s)ds\right|\right].
\end{equation*}
Now we proceed as in the proof of the replacement lemma at the bulk: by the entropy  and Jensen's inequalities, and Feynman--Kac formula the last expression is bounded above by
\begin{equation*}
 \frac{C}{N}+t\sup_{f}\left\{\int N^\delta\mcb L^{L,\star}_N \xi^A_1(\eta)f(\eta)d\nur+\frac{N}{B}\langle \mcb L_N\sqrt f,\sqrt f\rangle_{\nur}\right\}.
\end{equation*}
We apply Corollaries~\ref{cor:DirichletLB} and~\ref{cor:DirichletLR} to bound last display from above by
\begin{equation*}
\begin{aligned}
 &\frac{C}{N}+t\sup_{f}\Big\{\int N^\delta\mcb L^{L,\star}_N \xi^A_1(\eta)f(\eta)d\nur-\frac{N}{4B}\mcb D^B_N(\sqrt f,\nur)+\frac{C}{B}\\
 &\qquad\qquad\qquad\qquad\qquad+\frac{CN}{N^\theta}-\frac{N}{2B}\mcb D^L_N(\sqrt f,\nur)-\frac{N}{2B}\mcb D^R_N(\sqrt f,\nur)\Big\}.
\end{aligned}
 \end{equation*}
Using the definition of the adjoint $\mcb L^{L,\star}_N$ we can bound the integral as follows
 \begin{align}
 \left|\int N^\delta\mcb L^{L,\star}_N \xi^A_1(\eta)f(\eta)d\nur(\eta)\right|=&\Big|\int N^\delta \xi^A_1(\eta)\mcb L^{L}_Nf(\eta)d\nur \Big|\\
 \leq& \left|\int N^\delta c^{+}_1(\eta)\xi^A_1(\eta)[f(\eta^{1,+})-f(\eta)]d\nur(\eta)\right|\label{eq:term7}\\
  +&\left|\int N^\delta c^{-}_1(\eta)\xi^A_1(\eta)[f(\eta^{1,-})-f(\eta)]d\nur(\eta)\right|.\label{eq:term8}
 \end{align}

The strategies to bound the terms~\eqref{eq:term7} and~\eqref{eq:term8} are identical, so we write down only the proof for the first term. As we did in the bulk, we bound using Young's inequality:
\begin{align}
 \eqref{eq:term7}\leq& \frac{N^\delta A}{2} \left|\int c^{+}_1(\eta)[\sqrt f(\eta^{1,+})-\sqrt f(\eta)]^2d\nur(\eta)\right|\label{eq:term9}\\
+& \frac{N^\delta}{2A}\left|\int c^{+}_1(\eta)\xi^A_1(\eta)[\sqrt f(\eta^{1,+})+\sqrt f(\eta)]^2 d\nur(\eta)\label{eq:term10}\right|.
 \end{align}
If we choose $A=N/BN^\delta$,~\eqref{eq:term9} is killed by the Dirichlet form $\mcb D^L_N(\sqrt f,\nur)$ in the supremum, while~\eqref{eq:term10} is bounded by
$C \frac{BN^{2\delta}}{2N}\max{\{\tfrac{1}{N^\delta}, \tfrac{1}{N^\theta}\}}\leq CB\tfrac{N^\delta}{N}$, as long as $\int f(\eta^{1,+})d\nur(\eta)<\infty$, which is true since $f$ is a density.
Putting together all the estimates we get
\begin{equation}
 \E_{\mu_N}\left|\int_0^t N^\delta\mcb L^{L,\star}_N \xi^A_1(\eta)ds\right|\lesssim\frac{1}{N}+\frac{1}{B}+ \frac{BN^\delta}{N}+ \frac{N}{N^\theta},
\end{equation}
which goes to zero as $N\to\infty$ and $B\to\infty$.
 \end{proof}

\subsection{Energy estimate}\label{sec:energy}

To prove that the boundary terms of the hydrodynamic equation are well defined and that the space derivatives are well defined in the weak sense, we need to show that the limiting densities of the empirical measure belong to the Sobolev Space $L^2([0,T]\times \mathbb H)$ where $\mathbb H={\mcb H}_1(0,1)^3$, or equivalently that the density profile $\rho_t$ belongs to $\mathbb H$, almost surely in $t\in[0,T]$. 
\begin{prop}
 Let $\mathbb Q^\star$ be a limit point of the sequence of measures $\{\mathbb Q_N\}_N$. Then, the measure $\mathbb Q^\star$ is concentrated on paths $\rho(t,u)du$ such that $\rho\in L^2([0,T]\times \mathbb H)$, i.e.
 \begin{equation*}
 \mathbb{Q}^\star\big(\pi\in{\mcb D}([0,T],{\mcb M}^3):\rho\in L^2([0,T]\times\mathbb H)\big)=1.
\end{equation*}
\end{prop}

This follows from the next two lemmas together with Riesz's representation theorem. The strategy follows~\cite{FNG13},\cite{FNG13bis} and \cite{BMNS17}.
\begin{lem}
 Fix $\alpha\in\{A,B,\emptyset\}.$ For all $\theta\geq 1$ and $\theta\geq \delta$, there is a positive constant $\kappa>0$ such that
\begin{equation*}
 \E_{\mathbb Q^\star}\left[\sup_H\left\{\int_0^T \int_0^1 \partial_u H(s,u)\rho^\alpha_s(u) duds - \kappa\int_0^T \int_0^1 H(s,u)^2 duds\right\}\right]<\infty,
\end{equation*}
where the supremum is taken over all functions $H\in C^{0,2}_c([0, T] \times (0, 1))$.
\end{lem}
\begin{proof}
 Here we want to show that the linear functional
 \begin{equation*}
  \int_0^T \int_0^1 \partial_u H(t,u)\rho^\alpha_t(u) dtdu
 \end{equation*}
 is $\mathbb Q^\star$-almost surely continuous. If we consider a dense sequence $\{H_n\}_{n\geq1}$ in $C^{0,2}_c([0, T] \times (0, 1))$, it is sufficient to prove that, for any $\ell\geq1$,
 \begin{equation*}
 \E_{\mathbb Q^\star}\left[\max_{1\leq i\leq\ell}\left\{\int_0^T \int_0^1 \partial_u H_i(s,u)\rho^\alpha_s(u) duds - \kappa\int_0^T \int_0^1 H_i(s,u)^2 duds\right\}\right]<C,  
 \end{equation*}
for some constant $C>0$ independent from $\ell$. Since $\mathbb Q_N$ converges to $\mathbb Q^\star$, this is equivalent to showing that
\begin{equation}\label{eq:EE1}
 \lim_{N\to\infty}\E_{\mu_N}\left[\max_{1\leq i\leq\ell}\left\{\int_0^T  \langle\partial_u H_i(s,\cdot),\pi^{N,\alpha}_s\rangle  - \kappa \| H_i(s,\cdot)\|_{L^2(0,1)}^2ds\right\}\right]<C.
\end{equation}
As we did above in the replacement lemmas, we use entropy  and Jensen's inequality and the fact that $\exp \{\max_{1\leq i \leq \ell} a_i\}\leq \sum_{1\leq i\leq\ell} e^{a_i}$ to bound the expected value in~\eqref{eq:EE1} with
\begin{equation*}
 \frac{H(\mu_N|\nur)}{N}+\frac{1}{N}\log\sum_{1\leq i\leq \ell}\E_{\nur}\left[\exp\left\{N\int_0^T  \langle\partial_u H_i(s,\cdot),\pi^{N,\alpha}_s\rangle  - \kappa N \| H_i(s,\cdot)\|_{L^2(0,1)}^2 ds\right\}\right].
\end{equation*}
By Lemma~\ref{lem:entropy}, the first term is bounded by a constant $C_0$. For the second term it is enough to show that
\begin{equation*}
 \limsup_{N\to\infty}\frac{1}{N}\log\E_{\nur}\left[\exp\left\{N\int_0^T  \langle\partial_u H(s,\cdot),\pi^{N,\alpha}_s\rangle  - \kappa N \| H(s,\cdot)\|_{L^2(0,1)}^2 ds\right\}\right]<\tilde CT,
\end{equation*}
for a fixed $H$ and for a constant $\tilde C$ independent of $H$. The definition of empirical measure leads us to the following lemma.
\end{proof}
\begin{lem}
 Let $\rho$ be a profile satisfying \eqref{assu_1} and \eqref{assu_2}. For all $\theta\geq 1$ and $\theta\geq \delta$, there exists a positive constant $\kappa>0$ such that
\begin{equation*}
 \limsup_{N\to\infty}\frac1N \log\E_{\nur}\left[\exp \left\{ \int_0^T\sum_{x=1}^{N-1}\xi^\alpha_x(\eta_s) \partial_x H(s,\tfrac{x}{N}) - \kappa N \|H(s,\cdot)\|^2_{L^2(0,1)}ds\right\}\right]< CT.
\end{equation*}
\end{lem}
\begin{proof}
 By Feynman--Kac formula the expression in the statement is bounded above by
 \begin{equation*}
  \int_0^T \sup_f\left\{ \frac{1}{N}\int \sum_{x=1}^{N-1}\partial_u H(s,\tfrac{x}{N})\xi^\alpha_x(\eta)  f(\eta)d\nur + N \langle \mcb{L}_N \sqrt f, \sqrt f \rangle_{\nur} -\kappa \|H(s,\cdot)\|^2_{L^2(0,1)}\right\} ds.
 \end{equation*}
By Corollaries~\ref{cor:DirichletLB} and~\ref{cor:DirichletLR} we can bound the last display by
\begin{equation*}
  \int_0^T \sup_f\left\{\frac{1}{N}\int\sum_{x=1}^{N-1}\partial_u H(s,\tfrac{x}{N})\xi^\alpha_x(\eta)  f(\eta)d\nur -\frac{N}{4}{\mcb D}^B_N(\sqrt f,\nur)+\frac{CN}{N^{\theta}}+C -\kappa \|H(s,\cdot)\|^2_{L^2(0,1)}\right\} ds.
 \end{equation*}
 Using the fact that
 \begin{equation*}
  \partial_u H(s,\tfrac{x}{N})=H(s,\tfrac{x+1}{N})-H(s,\tfrac{x}{N})+\Or(N^{-1}),
 \end{equation*}
after a summation by parts, we can write the integral inside the supremum as
\begin{equation*}
 \sum_{x=1}^{N-2}\int H(s,\tfrac{x}{N})(\xi^\alpha_x(\eta) -\xi^\alpha_{x+1}(\eta) ) f(\eta)d\nur+\Or(N^{-1}).
 \end{equation*}
 Now we split the integral as its half plus its half and  we add and subtract the same term but replacing $f(\eta)$ with $ f(\eta^{x,x+1}) $ to obtain the term
 \begin{align}
 &\sum_{x=1}^{N-2}\int H(s,\tfrac{x}{N})(\xi^\alpha_x(\eta) -\xi^\alpha_{x+1}(\eta) ) [f(\eta)-f(\eta^{x,x+1})]d\nur\label{eq:EE2}\\
 +&\sum_{x=1}^{N-2}\int H(s,\tfrac{x}{N})(\xi^\alpha_x(\eta) -\xi^\alpha_{x+1}(\eta) )[f(\eta)+f(\eta^{x,x+1})]d\nur.\label{eq:EE3}
 \end{align}
 We perform the change of variables $\eta\mapsto\eta^{x,x+1}$ in~\eqref{eq:EE3} to get
 \begin{align}
 &\sum_{x=1}^{N-2}\int H(s,\tfrac{x}{N})(\xi^\alpha_x(\eta) -\xi^\alpha_{x+1}(\eta) ) [f(\eta)-f(\eta^{x,x+1})]d\nur\label{eq:EE4}\\
 +&\sum_{x=1}^{N-2}\int H(s,\tfrac{x}{N})(\xi^\alpha_x(\eta) -\xi^\alpha_{x+1}(\eta) )f(\eta^{x,x+1})\left(1-\frac{\nur(\eta^{x,x+1})}{\nur(\eta)}\right)d\nur.\label{eq:EE5}
 \end{align}
 We multiply and divide by $c_{x,x+1}(\eta)$. Then, by Young's inequality, for $A>0$, we can bound \eqref{eq:EE4} from above by
 \begin{align}
&\frac{A}{2}\sum_{x=1}^{N-2}\int (H(s,\tfrac{x}{N}))^2\frac{1}{c_{x,x+1}(\eta)}[\sqrt{f(\eta)}+\sqrt{f(\eta^{x,x+1})}]^2 d\nur\label{eq:EE6}\\
 +&\frac{1}{2A}\sum_{x=1}^{N-2}\int c_{x,x+1}(\eta)[\sqrt{f(\eta)}-\sqrt{f(\eta^{x,x+1})}]d\nur.\label{eq:EE7}
 \end{align}
 Choosing $A=1/2N$, last display  is bounded from above by
 \begin{equation}\label{eq:EE8}
  \frac{C}{N}\sum_{x=1}^{N-2}(H(s,\tfrac{x}{N}))^2 + \frac{N}{4}{\mcb D}^B_N(\sqrt f,\nur).
 \end{equation}
 For~\eqref{eq:EE5}, we first recall that by~\eqref{eq:diff_measures}
 \begin{equation*}
  \Big|1-\frac{\nur(\eta^{x,x+1})}{\nur(\eta)}\Big|\leq\frac{C}{N}.
 \end{equation*}
Again by Young's inequality, for $B>0$, we can bound  \eqref{eq:EE5} from above by
\begin{equation}\label{eq:EE9}
\begin{aligned}
&\frac{B}{2}\sum_{x=1}^{N-2}\int (H(s,\tfrac{x}{N}))^2 d\nur+\frac{1}{2B}\sum_{x=1}^{N-2}\int f^2(\eta^{x,x+1})\left(1-\frac{\nur(\eta^{x,x+1})}{\nur(\eta)}\right)^2 d\nur\\
 \leq&\frac{C}{N}\sum_{x=1}^{N-2}(H(s,\tfrac{x}{N}))^2 + C,
 \end{aligned}
\end{equation}
if we choose $B=1/2N$. Putting together~\eqref{eq:EE8} and~\eqref{eq:EE9}, we bound everything by
\begin{equation*}
 \int_0^T \sup_f\left\{\frac{C}{N}\sum_{x=1}^{N-2}(H(s,\tfrac{x}{N}))^2 -\kappa\int_0^1 (H(s,u))^2 du+ C+\frac{CN}{N^\theta}+\frac{C}{N}\right\} ds.
\end{equation*}
And, we can conclude by observing that
\begin{equation*}
 \frac{1}{N}\sum_{x=1}^{N-2}(H(s,\tfrac{x}{N}))^2 \to \int_0^1 (H(s,u))^2 du.
\end{equation*}
\end{proof}

Now we show some consequences of the previous result. 
\begin{lem}\label{lem:boundary_terms}
If $\rho^\alpha\in L^2([0,T] \times\mcb H^1(0,1)),$ then 
\begin{equation}\label{eq:boundary_terms}
\lim_{\epsilon\to 0}\Big|\rho^\alpha_s(u)-\tfrac{1}{\epsilon}\int_{u}^{u+\epsilon} \rho^\alpha_s(v)dv\Big| =0\quad \textrm{and }\quad \lim_{\epsilon\to 0}\Big|\rho^\alpha_s(u)-\tfrac{1}{\epsilon}\int_{u-\epsilon}^{u} \rho^\alpha_s(v)dv\Big| =0
\end{equation}
for all $u\in[0,1]$  and for a.e. $s\in[0,T]$.
\end{lem}
 \begin{proof}
 We present the proof for the limit on the left-hand side of last display but the other one is analogous. From the Cauchy-Schwarz inequality and the fact that $\rho^\alpha\in  L^2([0,T] \times\mcb  H^1(0,1)),$ we get that 
 \begin{equation*}\begin{split}
\Big|\rho^\alpha_s(u)-\tfrac{1}{\epsilon}\int_{u}^{u+\epsilon} \rho^\alpha_s(v)dv\Big|&\leq \tfrac{1}{\epsilon}\int_{u}^{u+\epsilon}\Big |\rho^\alpha_s(u)-\rho^\alpha_s(v)\Big|dv\leq  \tfrac{1}{\epsilon}\int_{u}^{u+\epsilon}\int_u^v \Big |\partial_q\rho^\alpha_s(q)\Big|dq \,dv\\ &\leq  \tfrac{\|\partial_q\rho^\alpha_s\|_2}{\epsilon}\int_{u}^{u+\epsilon} \sqrt {v-u}  \,dv=\tfrac{2}{3}\|\partial_q\rho^\alpha_s\|_2 \sqrt \epsilon.
\end{split}
\end{equation*}
Taking the limit as $\epsilon\to 0$ we are done. 
 \end{proof}

 \subsection{Properties of the solution}\label{section_properties}

 Now we prove that the solution $\rho^\alpha_t$ satisfies item $\textbf{D2}$ of Definition \ref{def:Dirichlet}  if $\delta<1\leq \theta$. Let $\mathbb Q$  be a limit point of the sequence $\lbrace\mathbb Q_{N}\rbrace_{N \geq 1}$ whose existence is a consequence of  Proposition \ref{prop:tight}. In fact,  we can assume that the  whole sequence $\lbrace\mathbb Q_{N}\rbrace_{ N \ge 1}$ converges to $\mathbb Q$.  Since we are dealing with a process that has at most one particle per site, we can conclude easily (for details we refer the reader to \cite{kipnisLandim}) that $\mathbb Q$ is supported on trajectories of measures are absolutely continuous with respect to the Lebesgue measure, i.e. $\pi^\alpha_t (du)=\rho^\alpha_{t}(u)du$ for any $t \in [0,T]$. In last section we  prove that  the density $\rho^\alpha_t$ belongs to $L^2 (0,T; {\mcb H}^1)$ and we can identify the profile $\rho^\alpha_t$ with a continuous function in $[0,1]$. To show that the profile satisfies $\textbf{D2}$ of Definition \ref{def:Dirichlet},  we use \eqref{eq:ident} and then,  for any $\delta>0$, 
$${\mathbb Q_N}  \left[\Big|\int_0^t\langle \pi^\alpha_s, \ora{\iota_\epsilon}(0) \rangle -r_\alpha\, ds\Big| >\delta \right] \; \le \; \delta^{-1} \, \mathbb E_{\mu_N}\left[\Big|\int_0^t\ora{\xi}^{\alpha,\epsilon N}_1(\eta_{s})-r_\alpha\, ds\Big|\right].$$  Now we want to apply Portmanteau theorem but since $\ora{\iota_\epsilon}(0)$ is not a continuous function, the set $\Big\{ \pi \, ;\, \Big|\int_0^t( \langle \pi^\alpha_s, \ora{\iota_\epsilon}(0) \rangle -r_\alpha)\, ds\Big| >\delta \Big\}$ is not an  open set in the Skorohod topology. To overcome the problem, we use an  $L^1$-approximation of $\ora{\iota_\epsilon}$ by continuous functions, to conclude  that
$${\mathbb Q}  \left[\Big|\int_0^t \langle \pi^\alpha_s, \ora{\iota_\epsilon}(0) \rangle -r_\alpha\, ds\Big| >\delta \right] \; \le \; \delta^{-1} \, \liminf_{N \to \infty} \, \mathbb E_{\mu_N}\left[\Big|\int_0^t\ora{\xi}^{\alpha,\epsilon N}_1(\eta_{s})-r_\alpha\, ds\Big|\right].$$
If the right-hand side of the previous inequality is zero,  since  $\mathbb Q$ a.s. $\pi^\alpha_s (du) = \rho^\alpha_s (u) du$ with $\rho^\alpha_s$ a continuous function at $0$ for  a.e. $s\in[0,T]$, by taking the limit $\epsilon \to 0$, we conclude that $\mathbb Q$ a.s. $\rho^\alpha_s (0) =r_\alpha$ for a.e. $s \in [0,T]$.  From the previous observations   we are left to proving that for  $\delta<1$ and for any $t\in[0,T]$ it holds
\begin{equation*}
\begin{split}
&{\limsup_{\epsilon \to 0}}\,{\limsup _{N\rightarrow \infty}}\,\mathbb E_{\mu_N}\left[\Big|\int_0^t \ora{\xi}^{\alpha,\epsilon N}_1(\eta_{s})-r_\alpha \, ds\Big|\right] =0.
\end{split}
\end{equation*}
But this is an easy consequence of Lemmas \ref{lem:local_RL} and \ref{lem:RL_boundary}. 

 \appendix

\section{Uniqueness of weak solutions}\label{sec:uniqueness}

In this appendix we show the uniqueness of solutions to the boundary value problems \eqref{eq:Dirichlet} and \eqref{eq:Robin_4}. We start by presenting an equivalent definition of a weak solution in the Dirichlet case, according to which we will show uniqueness.

\subsection{Equivalence of two notions of weak solution for the Dirichlet boundary value problem}

Recall  \eqref{eq:Dirichlet}.
In Definition \ref{dirichletA1} below, we give another notion of weak solutions to   \eqref{eq:Dirichlet}, and we prove that it is equivalent to Definition \ref{def:Dirichlet}. The advantage, in the proof of uniqueness of solutions, is that the definition below has as input test functions  in  $C^2_0([0,1])$ i.e.  twice continuously differentiable functions $\phi$ that vanish at the boundary, $\phi(0)=\phi(1)=0$.

\begin{defin}\label{dirichletA1}We say that  $\rho=(\rho^A,\rho^B,\rho^\emptyset ):[0,T]\times[0,1] \to [0,1]^3$ is a solution of~\eqref{eq:Dirichlet}
	if for each  $\alpha\in\{A,B,\emptyset\}$  it holds:
	\begin{itemize}
		\item [$\mathbf{D1.\textcolor{white}{'}}$]$\rho^\alpha \in L^2([0,T]\times \mcb H^1(0,1))$, 
		\item  [$\mathbf{D2'.}$]  for all $t\in [0,T]$ and any $\phi \in C_0^{2} ([0,1])$,
		\begin{equation}\label{9.3}\begin{split}
				\langle \rho^\alpha_{t},  \phi\rangle  -\langle \mathfrak g^\alpha,   \phi \rangle
				&- \int_0^t\langle \rho^\alpha_{s},\Delta \phi\rangle -\beta\langle \rho^\alpha_s(\rho_s^{\alpha+1}-\rho_s^{\alpha+2}),\nabla\phi\rangle ds\\&-\int_{0}^tr_\alpha\nabla \phi(0)-\tilde r_\alpha\nabla \phi(1) ds=0\end{split}
		\end{equation}
	\end{itemize}
\end{defin}
We will show that these two definitions are equivalent.
Clearly $\mathbf{D2'} \implies \mathbf{D3}$. 
It is not difficult to see that $\mathbf{D2'} \implies \mathbf{D2}$: to show $\rho_t^\alpha(0)=r_\alpha$ we consider in \eqref{9.3} smooth approximation of the test functions
$\phi^{(n)}(u)=u\textbf{1}_{[0,1/n]}-\tfrac{1}{n-1}(u-1)\textbf{1}_{[1/n,1]}$. We use integration by parts in the term with the Laplacian and then take the limit $n\to\infty$ to conclude. Now we will show that $
\{\mathbf{D1},\mathbf{D2},\mathbf{D3}\}\implies\{\mathbf{D1},\mathbf{D2'}\}$. 
Assume that $\rho$ satisfies $\{\textbf{D1},\textbf{D2},\textbf{D3}\}$. Fix $\phi \in C_0^{2} ([0,1])$ and take a sequence $\{\phi^{(n)}\}_{n\geq1}$ in $C_c^{2} ([0,1])$ converging to $\phi$ in $\mcb H^1_0$. 
Then $\textbf{D3}$ holds for $\phi^{(n)}$ in place of $\phi$. To conclude $\mathbf{D2'}$ we observe that
\begin{align}
	\lim_{n\to\infty}\langle \rho^\alpha_{t},  \phi^{(n)}\rangle = \langle \rho^\alpha_{t},   \phi\rangle, &\quad
	\lim_{n\to\infty}\langle \mathfrak g^\alpha,   \phi^{(n)} \rangle= \langle \mathfrak g^\alpha,   \phi\rangle\label{equiv1},\\
	\lim_{n\to\infty}\int_0^t\langle \rho^\alpha_{s},\Delta  \phi^{(n)}\rangle ds = \int_0^t\langle \rho^\alpha_{s},\Delta  \phi\rangle ds &+\int_{0}^t r_\alpha\nabla \phi(0)-\tilde r_\alpha\nabla \phi(1) ds\label{equiv2},\\
	\lim_{n\to\infty}\int_0^t\langle \rho^\alpha_s(\rho_s^{\alpha+1}-\rho_s^{\alpha+2}),\nabla\phi^{(n)}\rangle ds&= \int_0^t\langle \rho^\alpha_s(\rho_s^{\alpha+1}-\rho_s^{\alpha+2}),\nabla\phi\rangle ds.\label{equiv4}
\end{align}

The limits in \eqref{equiv1} and \eqref{equiv4} follows from Cauchy-Schwarz inequality. To show \eqref{equiv2}, we integrate by parts twice and use $\textbf{D2}$ to obtain
\begin{equation*}
	\lim_{n\to\infty}\langle\rho^\alpha_s,\Delta\phi^{(n)}\rangle=-\lim_{n\to\infty}\langle \nabla\rho^\alpha_s,\nabla\phi^{(n)}\rangle=-\langle \nabla\rho^\alpha_s,\nabla\phi\rangle=\langle\rho^\alpha_s,\Delta\phi\rangle+r_\alpha\nabla\phi(0)-\tilde r_\alpha \nabla\phi(1).
\end{equation*}

This concludes the proof of the equivalence of the two notions of weak solution.

\subsection{Uniqueness for Dirichlet boundary condition}\label{sec:uni_dir}

In this section we prove uniqueness of the solutions of the boundary value problem \eqref{eq:Dirichlet} (Dirichlet) following the strategy presented in Appendix 2.4 of~\cite{kipnisLandim}.

Consider $\rho^1$ and $\rho^2$ two weak solutions  of~\eqref{eq:Dirichlet} according to Definition \ref{dirichletA1} and  with the same initial condition and call $\bar\rho=\rho^1-\rho^2$ their difference. We know that, for $\alpha\in\{A,B,\emptyset\}$, $\bar\rho^\alpha\in L^2([0,T];\mcb H^1_0)$, where $\mcb H^1_0$ is the set of functions in $\mcb H^1$ vanishing at $0$ and $1$. I n fact, since $\rho^1$ and $\rho^2$ have the same boundary values, for a.e. $t\in[0,T]$, we have $\bar\rho^\alpha_t(0)=\bar\rho^\alpha_t(1)=0$.  
We consider the set $\{\psi_k\}$ of eigenfunctions of the Laplacian with Dirichlet boundary conditions, given by $\psi_k(u)=\sqrt{2} \sin(k\pi u)$, $k\geq1$. These functions are in $C^2_0([0,1])$ and they form an orthonormal basis of $L^2([0,1])$. For $\alpha\in\{A,B,\emptyset\}$, let us define
\begin{equation*}
	V_\alpha(t)=\sum_{k\geq1}\frac{1}{2a_k}\langle\bar\rho^\alpha_t,\psi_k\rangle^2
\end{equation*}
with $a_k=(k\pi)^2+1$. Let $V(t)=\sum_{\alpha}V_\alpha(t)$. To conclude the uniqueness we will show that there exists a constant $C>0$ such that $V'(t)\leq CV(t)$. Then, from Gronwall's inequality it follows that $V(t)=0$ for all $t\geq0$. This implies that for any $\alpha\in\{A,B,\emptyset\}$ and any $t\in[0,T]$, $\bar\rho^\alpha_t(u)=0$ almost every  $u\in[0,1]$, concluding the proof of uniqueness.

We have that
\begin{equation*}
	V_{\alpha}'(t)=\sum_{k\geq1}\frac{1}{a_k}\langle\bar\rho^\alpha_t,\psi_k\rangle\frac{d}{dt}\langle\bar\rho^\alpha_t,\psi_k\rangle,
\end{equation*}
and that from~\eqref{9.3}, 
\begin{equation*}
	\begin{aligned}
		\frac{d}{dt}\langle\bar\rho^\alpha_t,\psi_k\rangle=&\langle\bar\rho^\alpha_t,\Delta\psi_k\rangle-\beta\langle\Gamma^\alpha(\rho^1_t,\rho^2_t),\nabla\psi_k\rangle\\
		=&-k^2\pi^2\langle\bar\rho^\alpha_t,\psi_k\rangle-\beta\langle\Gamma^\alpha(\rho^1_t,\rho^2_t),\nabla\psi_k\rangle,
	\end{aligned}
\end{equation*}
where
\begin{equation}\label{A2}
	\Gamma^\alpha(\rho^1,\rho^2)=\rho^{1,\alpha}(\rho^{1,\alpha+1}-\rho^{1,\alpha+2})-\rho^{2,\alpha}(\rho^{2,\alpha+1}-\rho^{2,\alpha+2}).
\end{equation}
Then,
\begin{equation*}
	V_{\alpha}'(t)=-\sum_{k\geq1}\frac{(k\pi)^2}{a_k}\langle\bar\rho^\alpha_t,\psi_k\rangle^2-\sum_{k\geq1}\frac{\beta}{a_k}\langle\bar\rho^\alpha_t,\psi_k\rangle \langle\Gamma^\alpha(\rho^1_t,\rho^2_t),\nabla\psi_k\rangle.
\end{equation*}
From Young's inequality applied to the last term,
\begin{equation}\label{A4}
	\begin{aligned}
		V_{\alpha}'(t)&\leq-\sum_{k\geq1}\frac{(k\pi)^2}{a_k}\langle\bar\rho^\alpha_t,\psi_k\rangle^2+\frac{1}{2A}\sum_{k\geq1}\frac{\beta}{a_k}\langle\bar\rho^\alpha_t,\psi_k\rangle^2 +\frac{A}{2}\sum_{k\geq1}\frac{\beta}{a_k} \langle\Gamma^\alpha(\rho^1_t,\rho^2_t),\nabla\psi_k\rangle^2\\
		&=-\sum_{k\geq1}\frac{(k\pi)^2}{a_k}\langle\bar\rho^\alpha_t,\psi_k\rangle^2+\frac{\beta}{2A}\sum_{k\geq1}\frac{1}{a_k}\langle\bar\rho^\alpha_t,\psi_k\rangle^2 +\frac{A\beta}{2}\sum_{k\geq1}\frac{(k\pi)^2}{a_k} \langle\Gamma^\alpha(\rho^1_t,\rho^2_t),\varphi_k\rangle^2\\
		&\leq-\sum_{k\geq1}\frac{(k\pi)^2}{a_k}\langle\bar\rho^\alpha_t,\psi_k\rangle^2+\frac{\beta}{2A}\sum_{k\geq1}\frac{1}{a_k}\langle\bar\rho^\alpha_t,\psi_k\rangle^2 +\frac{A\beta}{2}\sum_{k\geq1} \langle\Gamma^\alpha(\rho^1_t,\rho^2_t),\varphi_k\rangle^2,
	\end{aligned}
\end{equation}
where in the last two displays we used the facts that $\nabla\psi_k(u)=k\pi\varphi_k(u)$ with $\varphi_k(u)=\sqrt{2}\cos(k\pi u)$, $k\geq1$ and that $a_k=(k\pi)^2+1$.
Since $\{\varphi_k; k\geq1\}$ is an orthonormal set in $L^2([0,1])$, 
\begin{equation}\label{A8}
	\sum_{k\geq1} \langle\Gamma^\alpha(\rho^1_t,\rho^2_t),\varphi_k\rangle^2\leq \|\Gamma^\alpha(\rho^1_t,\rho^2_t)\|_2^2,
\end{equation}
We can obtain an estimate on $|\Gamma^\alpha(\rho^1,\rho^2)|$ in the following way. We sum and subtract appropriate terms in \eqref{A2} to obtain
\begin{equation*}
	\Gamma^\alpha(\rho^1,\rho^2)=(\rho^{1,\alpha+1}-\rho^{1,\alpha+2})(\rho^{1,\alpha}-\rho^{2,\alpha})+\rho^{2,\alpha}(\rho^{1,\alpha+1}-\rho^{2,\alpha+1})-\rho^{2,\alpha}(\rho^{1,\alpha+2}-\rho^{2,\alpha+2}),
\end{equation*}
thus $|\Gamma^\alpha(\rho^1_t,\rho^2_t)|\leq \sum_\gamma|\bar\rho^{\gamma}_t|$, and then from Cauchy-Schwarz  inequality it follows that
\begin{equation}\label{A9}
	\|\Gamma^\alpha(\rho^1_t,\rho^2_t)\|_2^2\leq 3 \sum_{\gamma}\|\bar\rho^\gamma_t\|_2^2.
\end{equation}	
Using the bounds \eqref{A8} and \eqref{A9} in the last term of \eqref{A4} and then summing for $\alpha\in\{A,B,\emptyset\}$ we get
\begin{equation*}
	\begin{aligned}
		V'(t)&\leq\sum_{\alpha}\left[\sum_{k\geq1}\left(-\frac{(k\pi)^2}{a_k}+\frac{\beta}{2Aa_k}\right)\langle\bar\rho^\alpha_t,\psi_k\rangle^2+\frac{9A\beta}{2}\|\bar\rho^\alpha_t\|_2^2\right]\\
		&=\sum_{\alpha}\sum_{k\geq1}\left(-\frac{(k\pi)^2}{a_k}+\frac{\beta}{2Aa_k}+\frac{9A\beta}{2}\right)\langle\bar\rho^\alpha_t,\psi_k\rangle^2.
	\end{aligned}
\end{equation*}
If we choose $A=(9\beta/2)^{-1}$,
\begin{equation*}
	\begin{aligned}
		V'(t)&\leq\sum_{\alpha}\sum_{k\geq1}\left(-\frac{(k\pi)^2}{a_k}+\frac{9\beta^2}{4a_k}+1\right)\langle\bar\rho^\alpha_t,\psi_k\rangle^2=\left(\frac{9\beta^2}{2}+2\right)V(t),
	\end{aligned}
\end{equation*}
and this concludes the proof.

\subsection{Uniqueness for Robin boundary condition}\label{sec:uni_rob}

In the section we discuss the uniqueness of the solutions of the boundary value problems~\eqref{eq:Robin_4} (Robin), case b)  in Theorem \ref{th:hyd_ssep}.

First we observe that following a strategy similar to the one presented for the Dirichlet case, we can prove uniqueness for the Robin case~\eqref{eq:Robin_4} when $\kappa_1=\kappa_2=0$ (case $b1)$). We can adapt the strategy presented above by choosing $\{\varphi_k\}$ as orthonormal basis of $L^2([0,1])$. The weak formulation of equation~\eqref{eq:weak_rob_4} gives
\begin{equation*}
	\begin{aligned}
		\frac{d}{dt}\langle\bar\rho^\alpha_t,\varphi_k\rangle=&\langle\bar\rho^\alpha_t,\Delta\varphi_k\rangle-\beta\langle\Gamma^\alpha(\rho^1_t,\rho^2_t),\nabla\varphi_k\rangle+\nabla\varphi_k(0)\bar\rho^\alpha_t(0)-\nabla\varphi_k(1)\bar\rho^\alpha_t(1).
	\end{aligned}
\end{equation*}
Since $\nabla\varphi_k(u)=-k\pi\psi_k(u)$, the proof is concluded after we observe that $\nabla\varphi_k(0)=\nabla\varphi_k(1)=0$, for any $k$.

Now we study the cases $\kappa_1=\tilde\beta/2$, $\kappa_2=1$ (case b2)), and $\kappa_1=0$, $\kappa_2=1$ (case b3)). For the function $\bar\rho=\rho^1-\rho^2$, for each $\alpha\in\{A,B,\emptyset\}$ by ~\eqref{eq:weak_rob_4}
\begin{equation}\label{A18}\begin{aligned}
		0\,=\,&\langle \bar\rho^\alpha_{t},  \phi_{t}\rangle 
		- \int_0^t\langle \bar\rho^\alpha_{s},(\Delta + \partial_s) \phi_{s}\rangle ds+\beta\int_0^t\langle \Gamma^\alpha(\rho^1,\rho^2),\nabla\phi_s\rangle ds\\&-\int_{0}^t\nabla \phi_s(0)\bar\rho^\alpha_s(0)-\nabla \phi_s(1)\bar\rho^\alpha_s(1)ds\\&-\kappa_1\int_0^t  \phi_s(0){\left[(2r_{\alpha+2}-1)\bar\rho^\alpha_s(0)-2r_\alpha\bar\rho^{\alpha+1}_s(0)\right] }ds\\
		&-\kappa_1\int_0^t  \phi_s(1){\left[(1-2\tilde r_{\alpha+2})\bar\rho^\alpha_{s}(1)+2\tilde r_\alpha\rho_s^{\alpha+1}(1)\right]} ds\\&+\kappa_2\int_0^t \left[\phi_s(0)\bar\rho^\alpha_s(0)+\phi_s(1)\bar\rho^{\alpha}_s(1)\right]ds.\end{aligned}
\end{equation}
Approximating the function $\bar\rho^\alpha$ by a sequence $\{\phi^{(n)}\}$ of functions in $C^{1,2}([0,T]\times[0,1])$, we will assume that \eqref{A18} is true for $\phi=\bar\rho^\alpha$. This approximation could be formalized as in \cite[Theorem B.7]{FGLN22}. Then, after integration by part in the term with $(\Delta+\partial_s)\phi_s$ we obtain
\begin{equation}\label{A19}\begin{split}
		\frac{1}{2}\|\bar\rho^\alpha_{t}\|_2^2=& 
		- \int_0^t \|\nabla\bar\rho^\alpha_s\|_2^2\, ds-\beta\int_0^t\langle \Gamma^\alpha(\rho^1,\rho^2),\nabla\bar\rho^\alpha_s\rangle ds\\&+\int_0^t[\kappa_1(2r_{\alpha+2}-1)-\kappa_2]\bar\rho^\alpha_s(0)^2ds+\int_0^t[\kappa_1(1-2\tilde r_{\alpha+2})-\kappa_2]\bar\rho^\alpha_s(1)^2ds\\&-\int_0^t2\kappa_1r_\alpha\bar\rho^\alpha_s(0)\bar\rho^{\alpha+1}_s(0)ds+\int_0^t 2\kappa_1\tilde r_\alpha \bar\rho^{\alpha}_s(1)\bar\rho^{\alpha+1}_s(1) ds.\end{split}
\end{equation}
Define $W(t)=\frac{1}{2}\sum_{\alpha}\|\bar\rho^\alpha_{t}\|_2^2$. We will show that, for some constant $C$, $W(t)\leq C \int_0^t W(s)ds$ and then by Gronwall's inequality we will conclude that $W(t)=0$, for $t\in[0,T]$, which yields the uniqueness. 

Using Young's inequality in the second term of the right hand side of \eqref{A19}, summing in $\alpha$ and using \eqref{A9} we obtain
\begin{equation}\label{A16}\begin{split}
		W(t)\,\,\leq&\,\, 9A\beta\int_0^tW(s)\,ds+ 
		\left(\frac{\beta}{2A}-1\right) \int_0^t \sum_\alpha\|\nabla\bar\rho^\alpha_s\|_2^2\, ds\\&+\sum_{\alpha}\int_0^t[\kappa_1(2r_{\alpha+2}-1)-\kappa_2]\bar\rho^\alpha_s(0)^2ds+\sum_{\alpha}\int_0^t[\kappa_1(1-2\tilde r_{\alpha+2})-\kappa_2]\bar\rho^\alpha_s(1)^2ds\\&-\sum_{\alpha}\int_0^t2\kappa_1r_\alpha\bar\rho^\alpha_s(0)\bar\rho^{\alpha+1}_s(0)ds+\sum_{\alpha}\int_0^t 2\kappa_1\tilde r_\alpha \bar\rho^{\alpha}_s(1)\bar\rho^{\alpha+1}_s(1) ds.\end{split}
\end{equation}
Note that, in the case $\kappa_1=0$ and $\kappa_2=1$ (case b3)) the terms in the second line are negative and the terms in the last line disappear. So, taking $A>\beta/2$ we get $W(t)\leq 9A\beta \int_0^tW(s)ds$. Now let us consider the case $\kappa_1=\tilde\beta/2$ and $k_2=1$ (case b2))). To conclude the proof it remains to show that the sum of last two lines in \eqref{A16} is negative. Remember that, under the conditions of case b2), $\theta=\delta$ and the process is well defined only if $\tilde\beta<2$. We will conclude the uniqueness under the condition $\tilde\beta<4/3$.
By Young's inequality, for $D>0$
\begin{equation*}
\begin{aligned}
-2\bar\rho^\alpha_s(0)\bar\rho^{\alpha+1}_s(0)&\leq D \bar\rho^\alpha_s(0)^2+\frac{1}{D}\bar\rho^{\alpha+1}_s(0)^2\\
2\bar\rho^\alpha_s(1)\bar\rho^{\alpha+1}_s(1)&\leq \frac{1}{D} \bar\rho^\alpha_s(1)^2+D\bar\rho^{\alpha+1}_s(1)^2.
\end{aligned}
\end{equation*}
Taking $D=2$, we get that the sum of the last two lines in \eqref{A16} is bounded above by
\begin{equation*}
\begin{split}
\sum_{\alpha}\int_0^t\left[\frac{\tilde\beta}{2}\left(2(r_\alpha+r_{\alpha+2})-1+\frac{r_{\alpha+2}}{2}\right)-1\right]\bar\rho^\alpha_s(0)^2ds+\sum_{\alpha}\int_0^t\left[\frac{\tilde\beta}{2}\left(1+\frac{\tilde r_\alpha}{2}\right)-1\right]\bar\rho^\alpha_s(1)^2ds.
\end{split}
\end{equation*}
Using the bounds $r_\alpha+r_{\alpha+2}\leq1$ and $\tilde r_\alpha\leq1$ we see that the terms in the brackets above are negative when $\tilde\beta<4/3$, and this concludes the proof in this case.

\quad

\thanks{ {\bf{Acknowledgements: }}
P.G. thanks  FCT/Portugal for financial support through the project 
UID/MAT/04459/2013.   This project has received funding from the European Research Council (ERC) under  the European Union's Horizon 2020 research and innovative programme (grant agreement   n. 715734).
The work of A. Occelli was supported in part by ERC-2019-ADG Project 884584 LDRam and by  ERC-2016-STG Project 715734, and was partially developed while A.O. was a postdoctoral fellow at MSRI during the Program "Universality and Integrability in Random Matrix Theory and Interacting Particle Systems". The authors are grateful to Gunter Sch\"utz, Milton Jara, Rodrigo Marinho, C\'edric Bernardin, Leonardo De Carlo and Dan Betea for various  discussions about the model.}

\end{document}